\documentclass[11pt,a4paper,reqno]{amsart}

\usepackage{amsmath,amssymb,mathtools}
\usepackage{enumitem}
\usepackage{needspace}
\usepackage{microtype}
\usepackage[hidelinks,pdfusetitle]{hyperref}

\numberwithin{equation}{section}

\newtheorem{theorem}{Theorem}[section]
\newtheorem{proposition}[theorem]{Proposition}
\newtheorem{lemma}[theorem]{Lemma}
\newtheorem{corollary}[theorem]{Corollary}
\theoremstyle{definition}

\theoremstyle{remark}
\newtheorem{remark}[theorem]{Remark}

\newcommand{\E}{\mathbb E}
\newcommand{\Pp}{\mathbb P}
\newcommand{\R}{\mathbb R}
\newcommand{\N}{\mathbb N}
\newcommand{\cP}{\mathcal P}
\newcommand{\K}{\mathcal K}

\newcommand{\Law}{\mathcal L}
\newcommand{\1}{\mathbf 1}
\newcommand{\dd}{\,\mathrm d}
\newcommand{\TV}{\mathrm{TV}}

\DeclareMathOperator{\supp}{supp}

\title[ONE MARKED-SUM LAW DETERMINES MASS PARTITIONS]
{One marked-sum law determines random mass partitions and
\texorpdfstring{$\Xi$}{Xi}-coalescents}
\author{Jacopo Lenzi}
\thanks{ORCID: 0000-0003-2882-4223.}
\address{Department of Biomedical and Neuromotor Sciences\\
University of Bologna, Via San Giacomo 12, 40126 Bologna, Italy}
\email{jacopo.lenzi2@unibo.it}
\date{September 11, 2026}

\hypersetup{
  pdftitle={One marked-sum law determines random mass partitions and Xi-coalescents},
  pdfauthor={Jacopo Lenzi},
  pdfkeywords={random mass partition, Kingman simplex, determining mark
  distribution, coalescent with simultaneous multiple collisions,
  Fleming--Viot process}
}

\begin{document}
\raggedbottom

\begin{abstract}
Let $P=(P_j)$ be a random mass partition and let $(X_j)$ be iid real marks,
independent of $P$.  We construct a fixed mark distribution for which the
law of the real random variable $\sum_jP_jX_j$ determines the law of $P$ on
the Kingman simplex.  The observation map is an affine topological
embedding.  One construction uses an infinite convolution of stable
laws with rationally independent indices; for every prescribed
$q\in(0,\infty)$, a determining mark can instead be chosen symmetric,
centered, compound Poisson, and in $L^q$.  At any known positive time, the
law of the ranked block frequencies of a $\Xi$-coalescent started from
singletons---equivalently,
the exchangeable partition probability functions of all its finite
restrictions at that time---determines the finite
collision measure $\Xi$.  Composing the two results therefore identifies
$\Xi$ from the law of one real marked sum.  Within the $\Lambda$ subclass,
an asymmetric Bernoulli coloring determines the normalized collision
measure, as does any centered nonconstant mark whose absolute moment of
some order $a\in(1,4)$ is finite while every higher absolute moment is
infinite.  We also prove that, for every
$d\ge2$, a pair of distinct laws supported on strictly decreasing
$(d+1)$-tuples of positive masses with a fixed total mass remains
indistinguishable for every mark supported on at most $d$ points.  For every
fixed number $d\ge2$ of types, this obstruction yields distinct
normalized collision measures with the same mutation-free neutral $d$-type
$\Xi$-Fleming--Viot transition semigroup.
\end{abstract}

\maketitle

\section{Introduction}\label{sec:introduction}

We work on the Kingman simplex
\begin{equation}\label{eq:K-intro}
 \K=\left\{p_1\ge p_2\ge\cdots\ge0:
                   \sum_{j\ge1}p_j\le1\right\}
\end{equation}
and equip it with the coordinatewise product topology.
Given a random mass partition $P\sim\mu\in\cP(\K)$ and iid copies
$X_1,X_2,\ldots$ of an integrable real random variable $X$, independent of
$P$, consider
\begin{equation}\label{eq:TX-intro}
 M_X(P)=\sum_{j\ge1}P_jX_j,
 \qquad
 T_X(\mu)=\Law(M_X(P)).
\end{equation}
The series is absolutely convergent almost surely.  We call the distribution
of a centered $X\in L^1$ \emph{determining for random mass partitions} if
$T_X$ is injective on $\cP(\K)$.  The question is whether a single fixed
real mark distribution can be determining on the whole simplex.

Missing mass
$1-\sum_jP_j$ is the dust component, and atomic mass can become dust under
coordinatewise convergence.  Centered marks assign value zero to that
component and make $T_X$ continuous.  Dust is nevertheless determined once
the positive coordinates have been recovered, although it is not itself a
continuous coordinate.

This inverse problem concerns the mixing law $\mu$, rather than an unknown
mark distribution in a fixed weight model.  The latter direction has been
studied for normalized random measures by
Gaffi et~al.~\cite{GaffiLijoiPrunster2025}.  For random weighted averages,
Pitman~\cite[Corollary~9]{Pitman2018} proved, for proper mass partitions, that the
collection of output laws obtained by allowing all simple mark distributions
determines the associated partition structure.  Here the mark distribution
is held fixed, dust is allowed, and the input ranges over all laws on $\K$.
Via the paintbox--de Finetti correspondence, Bernoulli coloring is a
Bernoulli weighted sum of the paintbox masses, up to the dust
term \cite[Lemma~3.12]{SteifTykesson2019}.  A fixed Bernoulli parameter
$q\in(0,1)$
separates deterministic paintboxes \cite[Proposition~3.13]{SteifTykesson2019}, and for a parameter different from $1/2$ it identifies
residual-allocation laws generated by iid factors
\cite[Theorem~1.2]{BjornbergMaillerMortersUeltschi2020}.  Neither statement
gives injectivity on arbitrary mixing laws on $\K$.

A Gaussian mark is not determining even for deterministic proper mass
partitions: $(1/2,1/2)$ and $(2/3,1/6,1/6)$ both give the law $N(0,1/2)$
for standard Gaussian marks.  Our construction uses countably many distinct
power sums to retain information lost by a single stable component.
Choose rationally independent indices $\alpha_r\in(3/2,2)$ decreasing to
$3/2$ and positive coefficients $c_r$ that decrease sufficiently fast.  The
resulting mark has characteristic function
\begin{equation}\label{eq:stable-cf-intro}
 \phi_{X_*}(u)=
 \exp\left\{-\sum_{r\ge1}c_r|u|^{\alpha_r}\right\}.
\end{equation}
Conditionally on $P=p$, its marked sum has characteristic function
\begin{equation}\label{eq:stable-kernel-intro}
 \exp\left\{-\sum_{r\ge1}c_r|u|^{\alpha_r}V_{\alpha_r}(p)\right\},
 \qquad V_s(p)=\sum_{j\ge1}p_j^s.
\end{equation}
After the exponential is expanded, uniqueness of the Laplace transform of a
finite signed measure separates every mixed moment of the countable vector
$(V_{\alpha_r}(P))_{r\ge1}$.  It therefore recovers the joint law of these
power sums; their expectations alone need not determine the law of a random
mass partition.  This step remains valid although the
exponents have a finite accumulation point.  The profile itself determines
$p$: the holomorphic Dirichlet series $\sum_jp_j^z$ is known on a sequence
accumulating inside $\{\Re z>1\}$, and its successive largest terms recover
all positive masses and their multiplicities.

Lin~\cite[Theorem~2.2.1 in the arXiv version]{Lin2016} already uses rational
independence to recover mixed moments.  His observation consists of the
laws of all finite partial sums of an exchangeable sequence whose terms take
values in a rationally independent countable set with no finite accumulation point.
Here a single weighted-sum law suffices, and the exponents may accumulate.

\begin{theorem}[A determining mark from stable components]
\label{thm:intro-identifying}
There is a nonconstant, symmetric, centered and infinitely divisible random
variable $X_*$ such that
\[
 \E|X_*|^{3/2}<\infty,
 \qquad
 \E|X_*|^q=\infty\quad(q>3/2),
\]
and
\[
 T_{X_*}:\cP(\K)\longrightarrow\cP(\R)
\]
is an affine topological embedding for the weak topologies.
\end{theorem}

Compound Poisson marks allow the moment boundary to be placed beyond any
prescribed finite order.

\begin{theorem}[Prescribed integrability]\label{thm:compound-poisson}
For every $q\in(0,\infty)$, there is a nonconstant symmetric centered
compound Poisson random variable $X^{(q)}\in L^q$ such that
\[
 T_{X^{(q)}}:\cP(\K)\longrightarrow\cP(\R)
\]
is an affine topological embedding for the weak topologies.
\end{theorem}

The proof, given in Appendix~\ref{app:compound-poisson}, separates terms with
fractional exponents from an analytic remainder in the characteristic
exponent.  The mark is chosen separately for each prescribed $q$.

We next ask whether the distribution of a coalescent at one positive time
determines its collision measure.  Coalescents
allowing simultaneous mergers arise in the limit classification of
exchangeable population models by M{\"o}hle and Sagitov~\cite{MohleSagitov2001}.  We use the
collision-measure representation and rates of Schweinsberg~\cite{Schweinsberg2000}.  Let
$P^\Xi(t)$ be the ranked positive asymptotic block frequencies at time
$t>0$ of the homogeneous $\Xi$-coalescent started from singletons, with
missing mass representing dust.  By Kingman's paintbox representation
\cite{Kingman1978}, the law of $P^\Xi(t)$ is equivalent to the collection
of exchangeable partition probability functions (EPPFs) of all
finite restrictions at that same time.  Thus the observation in the next
theorem may be expressed either as a law on $\K$ or as all of those
finite-dimensional partition probabilities.

\begin{theorem}[One-time recovery of $\Xi$]\label{thm:intro-xi}
For every known $t>0$, the map
\[
 \Xi\longmapsto\Law(P^\Xi(t))
\]
is a topological embedding from the space of finite nonnegative collision
measures on $\K$ into $\cP(\K)$, with both spaces equipped with weak
convergence.  The atom at the zero partition represents the Kingman component.
\end{theorem}

The inversion uses one endpoint row of each finite restriction.  At level
$n$, the generator after the first collision is already determined by the
levels below $n$.  Variation of constants then expresses the observed row as
the unknown row of direct collision rates times a triangular matrix with
strictly positive diagonal.  Induction recovers the full collision-rate
array, and the representation theorem of Schweinsberg~\cite{Schweinsberg2000} recovers
$\Xi$.

Combining Theorems~\ref{thm:intro-identifying} and \ref{thm:intro-xi}
shows that, for known $t>0$, the law
\[
 \Law\left(\sum_{j\ge1}P_j^\Xi(t)X_j^*\right)
\]
determines the finite collision measure $\Xi$.  The pairwise restriction
also gives
\[
 1-\E V_2(P^\Xi(t))=\exp\{-t\Xi(\K)\}.
\]
Consequently an unknown time is recovered under the normalization
$\Xi(\K)=1$; without normalization the endpoint determines the product
$t\Xi$, which is the maximal possible conclusion under deterministic time
rescaling.

There is a uniform obstruction when the mark has bounded support
cardinality.  It is stronger than choosing a counterexample after seeing the
locations and probabilities of the atoms: the same pair of input laws works
for every mark distribution in the prescribed class.

\begin{proposition}[Finite-support obstruction]
\label{prop:intro-finite}
Fix $d\ge2$ and $s\in(0,1]$, and let
\begin{equation}\label{eq:strict-face}
 \mathcal F_{d+1,s}^{\circ}
 =\left\{p\in\K:\begin{array}{l}
 p_1>\cdots>p_{d+1}>0,\quad p_j=0\ (j>d+1),\\
 \sum_{i=1}^{d+1}p_i=s
 \end{array}\right\}.
\end{equation}
Let $U$ be a nonempty relatively open subset of this fixed-total-mass slice,
and let $\delta>0$.
There are distinct laws $H_{d,s,+}$ and $H_{d,s,-}$, compactly supported in
$U$, such that
\[
 T_X(H_{d,s,+})=T_X(H_{d,s,-})
\]
for every real mark distribution supported on at most $d$ points.  In the
affine coordinates of this slice, the two laws have $C_c^\infty$ densities
and may be chosen so that
$0<\|H_{d,s,+}-H_{d,s,-}\|_{\mathrm{TV}}<\delta$.
\end{proposition}

The construction uses the projection obstruction of
Cuesta-Albertos et~al.~\cite[Theorem~3.5]{CuestaAlbertosFraimanRansford2007}.
A Vandermonde polynomial vanishes whenever two of the $d+1$ coefficients
agree, as they must for marks taking at most $d$ values.  A perturbation
obtained by applying a differential operator then gives the smooth localized
pair.  For $d=2$, the simultaneous obstruction is already present in
Holroyd's three-mass example
\cite[Section~4.3]{BjornbergMaillerMortersUeltschi2020}: the same pair has
equal Bernoulli convolutions for every parameter, and hence equal marked-sum
laws for every two-point real mark.  The proposition gives the corresponding
obstruction for every $d\ge2$, with arbitrary localization, smooth compactly
supported densities and arbitrarily small positive total-variation separation.

The $\Lambda$ subclass provides a structured exception to the finite-support
obstruction.  If each block is colored by a Bernoulli variable with a fixed
known parameter $q\in(0,1)\setminus\{1/2\}$ and dust is colored at its base
frequency, the
law of the resulting frequency at one positive time determines the
normalized $\Lambda$ measure.  The proof recovers its Hausdorff moments one
at a time from a strict monotonicity property of the finite backward
equations.  Under unit-mass normalization the same law also determines an
unknown time.  Identification also holds for every centered nonconstant
mark whose moment abscissa lies in $(1,4)$ and whose moment at that abscissa
is finite (Corollary~\ref{cor:lambda-boundary-mark}).  In that case the
marked law determines the expected power-sum profile, which identifies
$\Lambda$ through the absorption probabilities of its finite restrictions.
Identification from an entire one-lag finite-type $\Lambda$-Fleming--Viot
transition operator was established by
Koskela et~al.~\cite[Lemma~1]{KoskelaJenkinsSpano2018} for probability densities supported
on $[\eta,1]$, bounded above and away from zero, with a fixed mutation
mechanism.  Here, in the mutation-free setting, one transition law of the
two-type process from the single known frequency
$q\in(0,1)\setminus\{1/2\}$ identifies an
arbitrary normalized $\Lambda$ measure.

The finite-support construction has a complementary consequence for the
forward process.  After tilting $H_{d,s,+}$ and $H_{d,s,-}$ by $V_2$, the
factor $V_2$ cancels the denominator in the $\Xi$-Fleming--Viot jump
generator.  The resulting distinct normalized collision measures generate
the same mutation-free neutral $d$-type transition semigroup
(Theorem~\ref{thm:xifv-obstruction}).  The same pair works for every initial
type frequency and every time: even the complete neutral dynamics with a
fixed number of types can fail to identify the collision measure.  By
contrast, a determining real mark identifies it from the distribution of
one marked sum at a single positive time.

The power-sum convergence principle in \cite{LenziBreiman} is used only
in Corollary~\ref{cor:lambda-boundary-mark}, where it recovers the
expected power sums needed for identification within the $\Lambda$
subclass.  The determining-mark constructions, the one-time recovery of
$\Xi$, and the finite-type obstruction are proved independently.

The paper is organized as follows.  Section~\ref{sec:operator} establishes the
topology of the observation map and the functional criterion.
Section~\ref{sec:stable-mark} constructs determining marks,
Section~\ref{sec:xi-fixed} gives the one-time $\Xi$ inversion, and
Section~\ref{sec:finite-support} proves the finite-support obstruction.
Section~\ref{sec:lambda} treats the $\Lambda$ subclass, and
Section~\ref{sec:xifv} proves the finite-type Fleming--Viot obstruction.  The
prescribed-moment construction and the four-label inversion are presented in
the appendices.

\section{The observation operator on the Kingman simplex}
\label{sec:operator}

We begin with the topology and the two elementary analytic facts used
throughout.  Equip $[0,1]^{\N}$ with the metric
\begin{equation}\label{eq:product-metric}
 d(p,p')=\sum_{j\ge1}2^{-j}|p_j-p_j'|.
\end{equation}
The set $\K$ in \eqref{eq:K-intro} is the intersection of the closed ordering
constraints $p_j\ge p_{j+1}$ and the closed constraints
$\sum_{j=1}^Np_j\le1$, $N\ge1$.  It is therefore compact and metrizable.
Its Borel $\sigma$-field is the trace of the product Borel $\sigma$-field
and is generated by the coordinate maps.  The space $\cP(\K)$ is consequently
compact and metrizable for weak convergence.

For a finite signed measure $\nu$, write
$\|\nu\|_{\TV}:=|\nu|(\Omega)$ for its total-variation norm on the
underlying space $\Omega$.  For two probability measures, this norm is twice
the usual total-variation distance.

Let $X$ be integrable, with characteristic function $\phi_X$.  For
$p\in\K$ and iid copies $(X_j)$,
\begin{equation}\label{eq:absolute-marked-series}
 \E\sum_{j\ge1}p_j|X_j|
 =\E|X|\sum_{j\ge1}p_j\le \E|X|.
\end{equation}
Thus $\sum_jp_jX_j$ converges absolutely almost surely.  Moreover,
\begin{equation}\label{eq:product-absolute}
 \sum_{j\ge1}|1-\phi_X(up_j)|
 \le |u|\E|X|\sum_{j\ge1}p_j<\infty,
\end{equation}
so its characteristic function is the order-independent product
\begin{equation}\label{eq:kernel}
 k_u^X(p)=\E\exp\left\{iu\sum_{j\ge1}p_jX_j\right\}
          =\prod_{j\ge1}\phi_X(up_j).
\end{equation}

\begin{lemma}[Continuity of the conditional kernel]
\label{lem:kernel-continuity}
If $X\in L^1$ and $\E X=0$, then $k_u^X\in C(\K)$ for every $u\in\R$.
Consequently $T_X:\cP(\K)\to\cP(\R)$ is affine and weakly continuous.
\end{lemma}

\begin{proof}
Differentiability of the characteristic function at the origin and centering
give
\[
 \eta(h):=\sup_{0<|v|\le h}\frac{|1-\phi_X(v)|}{|v|}
 \longrightarrow0\qquad(h\downarrow0).
\]
For a ranked subprobability sequence,
$(N+1)p_{N+1}\le1$.  Hence, uniformly in $p\in\K$,
\begin{align}
 \sum_{j>N}|1-\phi_X(up_j)|
 &\le |u|\eta\left(\frac{|u|}{N+1}\right)
       \sum_{j>N}p_j \notag\\
 &\le |u|\eta\left(\frac{|u|}{N+1}\right).
 \label{eq:uniform-product-tail}
\end{align}
For complex numbers of modulus at most one,
\[
 \left|\prod_jz_j-\prod_jw_j\right|\le\sum_j|z_j-w_j|
\]
for every finite product.  The same bound after passage to the limit, together
with \eqref{eq:uniform-product-tail}, shows that \eqref{eq:kernel} is a
uniform limit on $\K$ of continuous cylinder products.  It is therefore
continuous.

If $\mu_n\Rightarrow\mu$ in $\cP(\K)$, then
\[
 \int_\K k_u^X(p)\,\mu_n(\dd p)
 \longrightarrow
 \int_\K k_u^X(p)\,\mu(\dd p)
\]
for every $u$.  These are the characteristic functions of
$T_X(\mu_n)$ and $T_X(\mu)$, respectively, so L\'evy's continuity theorem
gives $T_X(\mu_n)\Rightarrow T_X(\mu)$.  Affinity follows directly from
mixing the law of $P$.
\end{proof}

Centering in Lemma~\ref{lem:kernel-continuity} cannot be dropped while retaining
the same state space and kernel.  Let $p^{(n)}$ have $n$ coordinates equal to
$1/n$ and all remaining coordinates zero.  Then $p^{(n)}\to0$
coordinatewise, but
\[
 \sum_jp_j^{(n)}X_j\longrightarrow \E X
\]
in probability.  If $\E X\ne0$, the kernel therefore need not converge to
$k_u^X(0)=1$.  This example also shows that neither the total atomic mass
$\sum_jp_j$ nor the dust mass $1-\sum_jp_j$ is continuous in the Kingman
topology.

The following standard separation argument gives an abstract criterion for a
determining mark.  The constructions below recover mixed power-sum moments
directly; the criterion gives the equivalent density property of their
conditional characteristic functions.

\begin{proposition}[Functional criterion for a determining mark]
\label{prop:functional-criterion}
Let $X\in L^1$ be centered.  Its distribution is determining for random
mass partitions if and only if
\begin{equation}\label{eq:functional-density}
 \overline{\operatorname{span}_{\R}
 \{1,\Re k_u^X,\Im k_u^X:u\in\R\}}^{\|\cdot\|_\infty}
 =C(\K;\R).
\end{equation}
\end{proposition}

\begin{proof}
Let $\mathcal H_X$ denote the real span in \eqref{eq:functional-density}.
If $\mathcal H_X$ is dense and $T_X(\mu)=T_X(\mu')$, equality of output
characteristic functions gives equality against $\mathcal H_X$, hence
against $C(\K;\R)$ by uniform approximation; Riesz representation gives
$\mu=\mu'$.

Conversely, suppose that the closure of $\mathcal H_X$ is proper.  Real
Hahn--Banach and Riesz give a nonzero finite signed regular measure $\sigma$
on $\K$ which annihilates $\mathcal H_X$.  Since $1\in\mathcal H_X$,
$\sigma(\K)=0$.  If $\sigma=\sigma^+-\sigma^-$ is its Jordan decomposition,
then
\[
 c:=\sigma^+(\K)=\sigma^-(\K)>0.
\]
The distinct probability measures $\nu_+=\sigma^+/c$ and
$\nu_-=\sigma^-/c$ agree against the real and imaginary parts of $k_u^X$
for every $u$.  Their output characteristic functions, and hence their
output laws, coincide.
\end{proof}

We shall use the power-sum coordinates on the Kingman simplex
\cite{Petrov2009},
\begin{equation}\label{eq:Vs}
 V_s(p)=\sum_{j\ge1}p_j^s,\qquad s>1.
\end{equation}
They are continuous on $\K$.  Indeed,
\begin{equation}\label{eq:Vs-tail}
 \sup_{p\in\K}\sum_{j>N}p_j^s
 \le (N+1)^{1-s},
\end{equation}
because $p_{N+1}\le(N+1)^{-1}$ and $\sum_{j>N}p_j\le1$.  Thus $V_s$ is a
uniform limit of its finite coordinate sums.  This uniform tail bound will
also supply local uniform convergence of the Dirichlet series used in the
next section.

\section{A determining mark from stable components}
\label{sec:stable-mark}

We now prove Theorem~\ref{thm:intro-identifying}.  The construction combines
stable components with different indices.  The arithmetic choice of the
indices is used only to separate mixed moments.

Set
\begin{equation}\label{eq:stable-exponents}
 a=\frac32,
 \qquad
 \alpha_r=\frac32+\frac14e^{-r},\qquad r\ge1.
\end{equation}
Let $(Z_r)_{r\ge1}$ be independent standard symmetric
$\alpha_r$-stable random variables, normalized by
\[
 \E e^{iuZ_r}=\exp\{-|u|^{\alpha_r}\}.
\]
Since $a<\alpha_r$, the norm $\|Z_r\|_a$ is finite.  Define
\begin{equation}\label{eq:stable-coefficients}
 b_r=\frac{2^{-r}}{1+\|Z_r\|_a},
 \qquad c_r=b_r^{\alpha_r},
 \qquad X_*=\sum_{r\ge1}b_rZ_r.
\end{equation}

\begin{lemma}[The mark and its moment boundary]
\label{lem:stable-construction}
The series in \eqref{eq:stable-coefficients} converges in $L^a$ and
absolutely almost surely.  Its limit is nonconstant, symmetric, centered and
infinitely divisible, with
\begin{equation}\label{eq:stable-cf}
 \phi_{X_*}(u)=
 \exp\left\{-\sum_{r\ge1}c_r|u|^{\alpha_r}\right\}.
\end{equation}
Moreover,
\begin{equation}\label{eq:stable-moments}
 \E|X_*|^a<\infty,
 \qquad
 \E|X_*|^q=\infty\quad\text{for every }q>a.
\end{equation}
The family $(\alpha_r)_{r\ge1}$ is linearly independent over $\mathbb Q$.
\end{lemma}

\begin{proof}
First suppose that a finite rational relation
$\sum_{r=1}^Rq_r\alpha_r=0$ holds.  With $z=e^{-1}$,
\[
 \frac32\sum_{r=1}^Rq_r+\frac14\sum_{r=1}^Rq_rz^r=0.
\]
This is a polynomial relation over $\mathbb Q$ for the transcendental number
$z$.  Its positive-degree coefficients, and hence all $q_r$, vanish.  This
proves rational independence.

By the choice of $b_r$,
\[
 \sum_{r\ge1}\|b_rZ_r\|_a\le\sum_{r\ge1}2^{-r}<\infty.
\]
Minkowski's inequality and completeness give convergence in $L^a$.  Also
$\sum_r b_r\E|Z_r|<\infty$, so Tonelli gives absolute convergence almost
surely.  Symmetry is inherited from the partial sums, and $L^a$ convergence,
with $a>1$, passes their zero means to the limit.

Since $0<b_r<1$ and $\alpha_r>1$, we have $c_r\le b_r$ and
$\sum_rc_r<\infty$.  Passing to the limit in the characteristic functions of
the finite sums gives \eqref{eq:stable-cf}; its exponent is continuous at the
origin because
\begin{equation}\label{eq:stable-exp-uniform}
 \sum_rc_r|u|^{\alpha_r}
 \le \max\{|u|^a,|u|^2\}\sum_rc_r.
\end{equation}
The finite stable convolutions are infinitely divisible and their laws
converge weakly.  Equivalently, the exponent in \eqref{eq:stable-cf} is a
locally uniform sum of continuous negative-definite functions.  Either
description shows that $X_*$ is infinitely divisible.  It is nonconstant
because the exponent is positive away from the origin.

The finiteness of the $a$th moment follows from the $L^a$ construction.  If
$q>a$,
choose $r$ such that $\alpha_r<q$ and write
$X_*=b_rZ_r+Y_r$, where the two summands are independent.  Choose $M<\infty$
such that $\Pp(|Y_r|\le M)>0$.  On
$\{|b_rZ_r|>2M,\ |Y_r|\le M\}$,
$|X_*|\ge |b_rZ_r|/2$.  Independence and the infinite $q$th moment of
$Z_r$ yield
\[
 \E|X_*|^q
 \ge 2^{-q}\Pp(|Y_r|\le M)
       \E\big[|b_rZ_r|^q\1_{\{|b_rZ_r|>2M\}}\big]
 =\infty.
\]
Indeed, a nondegenerate symmetric $\alpha_r$-stable law has finite absolute
$q$th moment exactly when $q<\alpha_r$; for $q\ge\alpha_r$, the
L\'evy-measure criterion in Sato~\cite[Theorem~25.3]{Sato1999} gives divergence.
\end{proof}

For $p\in\K$, nonnegative Tonelli and \eqref{eq:stable-exp-uniform} give
the conditional characteristic function
\begin{equation}\label{eq:stable-kernel}
 k_u^{X_*}(p)
 =\exp\left\{-\sum_{r\ge1}c_r|u|^{\alpha_r}
                         V_{\alpha_r}(p)\right\}.
\end{equation}
The exponent converges uniformly over $p\in\K$, since
$0\le V_{\alpha_r}(p)\le1$.

We next isolate the uniqueness statement that permits accumulating exponents.

\begin{lemma}[Finite signed Laplace uniqueness]
\label{lem:signed-laplace}
Let $\eta$ be a finite signed Borel measure on $[0,\infty)$.  If
\[
 \int_{[0,\infty)}e^{-s\lambda}\,\eta(\dd\lambda)=0
 \qquad\text{for every }s>0,
\]
then $\eta=0$.
\end{lemma}

\begin{proof}
Weight $\eta$ by $e^{-\lambda}$ and push the resulting finite signed
measure forward under $x=e^{-\lambda}$ to a measure $\rho$ on $[0,1]$.
For every integer $n\ge0$,
\[
 \int_{[0,1]}x^n\,\rho(\dd x)
 =\int_{[0,\infty)}e^{-(n+1)\lambda}\,\eta(\dd\lambda)=0.
\]
Polynomials are uniformly dense in $C[0,1]$, so $\rho=0$.  The map
$\lambda\mapsto e^{-\lambda}$ is a Borel isomorphism from $[0,\infty)$
onto $(0,1]$.  Hence
$e^{-\lambda}\eta(\dd\lambda)=0$; multiplying by the bounded function
$e^\lambda$ on $[0,L]$ and then letting $L\uparrow\infty$ gives
$\eta=0$.
\end{proof}

The next lemma extracts the law of the entire fractional power-sum profile
from one output distribution.  Write
\[
 v(p)=\big(V_{\alpha_r}(p)\big)_{r\ge1}\in[0,1]^{\N}.
\]

\begin{lemma}[Extraction of the profile law]
\label{lem:profile-extraction}
If $\mu,\mu'\in\cP(\K)$ and
$T_{X_*}(\mu)=T_{X_*}(\mu')$, then
$v_\#\mu=v_\#\mu'$.
\end{lemma}

\begin{proof}
Let $\rho=v_\#\mu$, $\rho'=v_\#\mu'$ and $\sigma=\rho-\rho'$.  Equality
of the output characteristic functions and \eqref{eq:stable-kernel} imply
\begin{equation}\label{eq:profile-transform-equality}
 \int\exp\left\{-\sum_{r\ge1}c_rt^{\alpha_r}x_r\right\}
       \sigma(\dd x)=0,
 \qquad t>0.
\end{equation}
For a finite-support multiindex
$m=(m_r)\in\mathbb N_0^{(\mathbb N)}$, put
\[
 |m|=\sum_rm_r,
 \quad m!=\prod_rm_r!,
 \quad c^m=\prod_rc_r^{m_r},
 \quad \lambda_m=m\mathbin{\cdot}\alpha=\sum_rm_r\alpha_r,
\]
and, for $m\ne0$, set
\begin{equation}\label{eq:stable-am}
 A_m=\frac{(-1)^{|m|}c^m}{m!}
       \int\prod_rx_r^{m_r}\,\sigma(\dd x).
\end{equation}
The coefficient corresponding to the zero multiindex is zero because
$\sigma$ has total mass zero.  The multiindex expansion is absolutely
summable:
\begin{equation}\label{eq:stable-TV-bound}
 \sum_m|A_m|
 \le \|\sigma\|_{\TV}\exp\left\{\sum_rc_r\right\}<\infty.
\end{equation}
Consequently, substituting $t=e^{-s}$ in
\eqref{eq:profile-transform-equality} gives
\[
 0=\sum_{m\ne0}A_me^{-s\lambda_m}
   =\int_{[0,\infty)}e^{-s\lambda}\,\eta(\dd\lambda),
 \qquad
 \eta=\sum_{m\ne0}A_m\delta_{\lambda_m}.
\]
The bound \eqref{eq:stable-TV-bound} makes $\eta$ a finite signed measure,
even though its atoms have finite accumulation points.  By
Lemma~\ref{lem:signed-laplace}, $\eta=0$.

Rational independence of $(\alpha_r)$ makes
$m\mapsto\lambda_m$ injective.  Thus
$\eta(\{\lambda_m\})=A_m=0$ for every nonzero finite-support multiindex.
All mixed cylinder moments of $\rho$ and $\rho'$ agree.  Cylinder
polynomials form a unital real algebra that separates points of the compact
cube $[0,1]^{\N}$, so Stone--Weierstrass gives $\rho=\rho'$.
\end{proof}

It remains to prove that the profile map is injective.

\begin{lemma}[Recovery from accumulating power sums]
\label{lem:profile-injective}
Let $(s_r)_{r\ge1}$ be distinct real numbers greater than one that converge
to $s_\infty>1$.  Then
\[
 v_s:\K\longrightarrow[0,1]^{\N},
 \qquad v_s(p)=(V_{s_r}(p))_{r\ge1},
\]
is continuous and injective, and hence is a homeomorphism onto its image.
\end{lemma}

\begin{proof}
Continuity of every coordinate follows from \eqref{eq:Vs-tail}.  For
$p\in\K$ define
\begin{equation}\label{eq:partition-dirichlet}
 F_p(z)=\sum_{j:p_j>0}p_j^z,
 \qquad \Re z>1.
\end{equation}
On every compact subset of this half-plane, the series converges locally
uniformly by \eqref{eq:Vs-tail}, with $s$ replaced by the infimum of the real
parts.  Hence $F_p$ is holomorphic.  If $v_s(p)=v_s(p')$, then
$F_p(s_r)=F_{p'}(s_r)$ for every $r$.  These points accumulate at
$s_\infty$, an interior point of the domain, so the identity theorem gives
$F_p=F_{p'}$ on $\{\Re z>1\}$.

We recall how the positive masses are recovered from this Dirichlet series.
The zero partition is characterized by $F_p\equiv0$.  Otherwise,
\begin{equation}\label{eq:largest-mass}
 p_1=\lim_{s\to\infty}F_p(s)^{1/s};
\end{equation}
indeed $p_1^s\le F_p(s)\le p_1^{s-1}\sum_jp_j\le p_1^{s-1}$.
Let $m_1$ be the multiplicity of $p_1$.  It is finite.  If further positive
masses remain, put $\theta=p_{m_1+1}/p_1<1$.  Then
\begin{align*}
 0\le \frac{F_p(s)}{p_1^s}-m_1
 &=\sum_{j>m_1}\left(\frac{p_j}{p_1}\right)^s\\
 &\le \theta^{s-1}\sum_{j>m_1}\frac{p_j}{p_1}
 \longrightarrow0.
\end{align*}
Thus $m_1$ is recovered.  Subtract $m_1p_1^z$ from $F_p(z)$ and repeat.
The recursion either terminates or recovers, at each finite stage, the next
positive mass and its finite multiplicity.  It therefore recovers every
positive coordinate of $p$, including infinite strictly decreasing
sequences and repeated masses.  Hence $p=p'$.

Finally, a continuous injection from compact $\K$ into the Hausdorff cube is
a homeomorphism onto its image.
\end{proof}

\begin{proof}[Proof of Theorem~\ref{thm:intro-identifying}]
The properties and moments of $X_*$ are given by
Lemma~\ref{lem:stable-construction}.  If two input laws have the same output law,
Lemma~\ref{lem:profile-extraction} gives equality of their profile pushforwards.
By Lemma~\ref{lem:profile-injective}, applied with $s_r=\alpha_r$, $v$ has a
Borel inverse on its compact image.  Consequently,
\[
 \mu=(v^{-1})_\#(v_\#\mu)
     =(v^{-1})_\#(v_\#\mu')=\mu'.
\]
Thus $T_{X_*}$ is injective.  It is
affine and weakly continuous by Lemma~\ref{lem:kernel-continuity}.  Since
$\cP(\K)$ is compact and $\cP(\R)$ is Hausdorff, it is a topological
embedding onto its image.
\end{proof}

For $p\in\K$, write
\begin{equation}\label{eq:deterministic-output-law}
 Q_p^{X_*}=\Law\left(\sum_{j\ge1}p_jX_j^*\right).
\end{equation}

\begin{corollary}[Affine representation and random measures]
\label{cor:affine-random-measure}
The compact convex set $T_{X_*}(\cP(\K))$ has compact extreme boundary
$\{Q_p^{X_*}:p\in\K\}$, and every law in this image has a unique mixture
representation in the laws $Q_p^{X_*}$ associated with deterministic mass
partitions $p\in\K$.  Thus this image is
a Bauer simplex.

Let $H$ be a known diffuse probability measure on a standard Borel space
$E$, and let $f:E\to\R$ be measurable with
$f_\#H=\Law(X_*)$.  If, conditionally on $P$, the variables $Z_j$ are iid
with law $H$ and
\[
 M=\left(1-\sum_{j\ge1}P_j\right)H
   +\sum_{j\ge1}P_j\delta_{Z_j},
\]
then the law of $M(f):=\int_E f(z)\,M(\dd z)$ determines $\Law(P)$ and hence
$\Law(M)$.
\end{corollary}

\begin{proof}
The space $\cP(\K)$ is a Bauer simplex with extreme boundary
$\{\delta_p:p\in\K\}$.  Theorem~\ref{thm:intro-identifying} gives an affine
homeomorphism from this simplex onto its image and sends $\delta_p$ to
$Q_p^{X_*}$.  In particular,
\[
 T_{X_*}(\mu)=\int_\K Q_p^{X_*}\,\mu(\dd p),
\]
and the representing law $\mu$ is unique.  This proves the first assertion.
For the second, centering gives
$\int f\,\dd H=\E X_*=0$, and therefore
\[
 M(f)=\sum_{j\ge1}P_jf(Z_j).
\]
Conditionally on $P$, the variables $f(Z_j)$ are iid with law
$\Law(X_*)$.  Thus $\Law(M(f))=T_{X_*}(\Law(P))$, which determines
$\Law(P)$.  Since $H$ is known, this law also determines the conditional
mixture construction, and hence $\Law(M)$.
\end{proof}

\section{Fixed-time recovery of a \texorpdfstring{$\Xi$}{Xi}-coalescent}
\label{sec:xi-fixed}

We use Schweinsberg's convention
\cite[Theorem~2 and formula~(11)]{Schweinsberg2000} throughout:
\(\Xi=a\delta_0+\Xi_0\), where \(a=\Xi(\{0\})\) and
\(\Xi_0(\{0\})=0\).  Thus the atom at zero is the Kingman component.  A
specified collision of \(r\) groups of sizes \(k_1,\ldots,k_r\geq2\),
leaving \(s\) singleton blocks, with \(b=k_1+\cdots+k_r+s\), has rate
\begin{equation}
\begin{split}
 \lambda^\Xi_{b;k_1,\ldots,k_r;s}
 &=a\mathbf 1_{\{r=1,\,k_1=2\}}\\
 &\quad+\int_{\K\setminus\{0\}}
 \sum_{\ell=0}^{s}\binom{s}{\ell}(1-|x|)^{s-\ell}\\
 &\hspace{2.8cm}\times
 \sum_{\substack{i_1,\ldots,i_{r+\ell}\\\mathrm{all\ distinct}}}
 x_{i_1}^{k_1}\cdots x_{i_r}^{k_r}
 x_{i_{r+1}}\cdots x_{i_{r+\ell}}\,
 \frac{\Xi_0(dx)}{V_2(x)},
 \label{eq:xi-schweinsberg-rates}
\end{split}
\end{equation}
where \(|x|=\sum_i x_i\).  This convention is important for both the
normalization and weak limits at the origin.

Related inversions use different observations.
For normalized collision measures, Schweinsberg~\cite[Proposition~6]{Schweinsberg2000}
reconstructs the measure from the induced merger partition of the other
blocks when two specified labels first coalesce, conditional on the number
of blocks immediately before that event.  For $\Lambda$-coalescents, M{\"o}hle~\cite[Lemma~6.3]{Mohle2008}
proves determination from the marginal laws of the times to the most recent
common ancestor.  We use the complete partition law at one deterministic
time.

Let \(P^\Xi(t)\) be the ranked positive asymptotic block frequencies of the
standard \(\Xi\)-coalescent started from singletons; its missing mass is
dust.  Knowing its law means knowing the probability distribution of this
random mass partition, as would be available from replications, rather than
one realization.  By Kingman's paintbox theorem, this law is equivalent to
the EPPF of the restriction at time $t$ for every sample size $n$.

For $n\geq1$, let $\mathcal S_n$ be the set of partitions of $[n]$, let
$0_n$ be its singleton partition, and put
$D_n=\mathcal S_n\setminus\{0_n\}$.  Write $q^{(n)}_{\rho,\sigma}$ for the
generator entry from $\rho$ to $\sigma$, and set
\[
 R_n=-q^{(n)}_{0_n,0_n},\qquad
 r_n=(q^{(n)}_{0_n,\pi})_{\pi\in D_n},
 \qquad R_1=0.
\]
In every $D_n$ we use a linear extension of refinement, with finer
partitions first.

\begin{lemma}[Projective determination of the post-jump generator]
\label{lem:xi-projective-subgenerator}
Suppose that $r_b$ and $R_b$ are known for every $b<n$.  Then the
subgenerator $B_n=(q^{(n)}_{\rho,\sigma})_{\rho,\sigma\in D_n}$ is known.
More explicitly, let $\rho\in D_n$ have blocks
$C_1,\ldots,C_b$, listed by increasing least element.  If $\sigma$ is a
strict coarsening of $\rho$, let $\theta_{\rho,\sigma}\in\mathcal S_b$ be
the labeled partition in which $i$ and $j$ are in the same block exactly
when $C_i$ and $C_j$ lie in the same block of $\sigma$.  Then
\begin{equation}
 q^{(n)}_{\rho,\sigma}=r_b(\theta_{\rho,\sigma}),
 \qquad q^{(n)}_{\rho,\rho}=-R_b,
 \label{eq:xi-projective-generator}
\end{equation}
and $q^{(n)}_{\rho,\sigma}=0$ when $\sigma$ is not a coarsening of $\rho$.
Consequently $B_n$ is triangular in the stated order.
\end{lemma}

\begin{proof}
The consistency and exchangeability of the coalescent imply that, from a
state with $b$ blocks, the current blocks evolve as a coalescent on $[b]$.
The canonical ordering of the blocks of $\rho$ identifies a particular
transition $\rho\to\sigma$ with the single labeled coordinate
$0_b\to\theta_{\rho,\sigma}$, which proves the off-diagonal identity in
\eqref{eq:xi-projective-generator}.  There is no multiplicity factor:
distinct labeled coarsenings index distinct coordinates of $r_b$;
multiplicities enter only when those coordinates are summed to form $R_b$.
The diagonal entry is therefore $-R_b$, and coalescents have no transitions
to a finer or incomparable partition.  Since every $\rho\in D_n$ has at
most $n-1$ blocks, all entries on the right-hand side are among the assumed
lower-level data.
\end{proof}

Recovering a generator from a transition matrix is the classical Markov
embedding problem \cite{IsraelRosenthalWei2001}.  The lemma shows why one
row suffices here: restrictions with fewer blocks determine the generator
after the first collision.  The remaining step is a triangular linear
inversion.  For the coupon-collection chain on the Boolean lattice,
Baake and Baake~\cite{BaakeBaake2026} use incidence-algebra methods to characterize its
continuous-time embeddings.

\begin{theorem}
\label{thm:xi-fixed-time}
For every known \(t>0\), the map
\[
 \Xi\longmapsto\mathcal L(P^\Xi(t))
\]
is a topological embedding from the space of finite nonnegative collision
measures on $\K$ into $\cP(\K)$, with both spaces equipped with the weak
topology induced by the product topology on $\K$.
Moreover,
\begin{equation}
 1-\E V_2(P^\Xi(t))=e^{-t\Xi(\K)}.
 \label{eq:xi-pair-identity}
\end{equation}
Thus an unknown positive time is also recovered on \(\Xi(\K)=1\).  In
general the endpoint law identifies exactly \(t\Xi\); when \(\Xi=0\), time
itself is not identifiable.
\end{theorem}

\begin{proof}
We first pass from the ranked-frequency law to finite labeled restrictions.
For a partition \(\pi\) of \([n]\), the probability, conditional on
\(P^\Xi(t)=p\), that the paintbox partition is coarser than \(\pi\) is
\begin{equation}
 F_\pi(p)=\prod_{B\in\pi:\ |B|\geq2}V_{|B|}(p).
 \label{eq:xi-paintbox-coarser}
\end{equation}
Here ``coarser'' is in the refinement order: every block of $\pi$ is
contained in one paintbox block.  M\"obius inversion on the finite partition
lattice recovers every exact labeled restriction probability from these
probabilities, in accordance with Kingman's paintbox representation
\cite{Kingman1978}.  Each \(V_k\) is continuous by \eqref{eq:Vs-tail}.
Hence the frequency law determines the complete endpoint row
\[
 p_n(t)=\bigl(\mathbb P_{0_n}\{\Pi_n^\Xi(t)=\pi\}\bigr)_{\pi\in\mathcal S_n}
\]
for every \(n\).  Zero-frequency blocks are almost surely singleton blocks
in an exchangeable paintbox, so no nonsingleton information is hidden in the
dust.

We use row vectors, so
\(p_n'(t)=p_n(t)Q_n\), and write
\begin{equation}
 Q_n=\begin{pmatrix}-R_n&r_n\\0&B_n\end{pmatrix}.
 \label{eq:xi-block-generator}
\end{equation}
Here \(r_n\) is the row of direct rates from the singleton partition and
\(B_n\) is the subgenerator on \(D_n\).  By
Lemma~\ref{lem:xi-projective-subgenerator}, $B_n$ is determined coordinatewise
by the direct rows already recovered at levels $b<n$.  Since the chain cannot
return to \(0_n\),
\begin{equation}
 p_n(t;0_n)=e^{-R_nt},\qquad R_n=-t^{-1}\log p_n(t;0_n).
 \label{eq:xi-no-jump}
\end{equation}
Variation of constants, with the displayed row orientation, gives
\begin{equation}
 p_n(t;D_n)=r_nK_n,
 \qquad
 K_n=\int_0^t e^{-R_ns}e^{(t-s)B_n}\,ds.
 \label{eq:xi-one-row-inversion}
\end{equation}
The matrix \(K_n\) is triangular.  If \(\pi\in D_n\) has \(b\) blocks,
its diagonal entry is
\begin{equation}
 (K_n)_{\pi\pi}=\int_0^te^{-R_ns}e^{-R_b(t-s)}\,ds>0.
 \label{eq:xi-volterra-diagonal}
\end{equation}
When \(R_n=R_b\), this equals \(te^{-R_nt}\), rather than a singular
quotient.  Therefore \(K_n\) is invertible and
\begin{equation}
 r_n=p_n(t;D_n)K_n^{-1}.
 \label{eq:xi-rate-recovery}
\end{equation}
Starting at \(n=2\), this recovers every specified collision rate.  The
representation in Theorem~2 and formula~(11) of Schweinsberg~\cite{Schweinsberg2000},
together with its uniqueness assertion in Proposition~4, then recovers
\(\Xi\) from the complete array \eqref{eq:xi-schweinsberg-rates}.

The first few restrictions are as follows.  For \(n=2\),
\(R_2=\lambda^\Xi_{2;2;0}=\Xi(\K)\), and
\[
 K_2=\int_0^te^{-R_2s}\,ds=
 \begin{cases}(1-e^{-R_2t})/R_2,&R_2>0,\\ t,&R_2=0.\end{cases}
\]
For \(n=3\), put \(u=\lambda^\Xi_{3;2;1}\),
\(v=\lambda^\Xi_{3;3;0}\), and \(R_3=3u+v\).  Then, for each labeled
pair $ij$ and remaining label $k$,
\[
 p_3(t;ij\mid k)=u\kappa_{3,2},\qquad
 p_3(t;123)=3u(\kappa_{3,1}-\kappa_{3,2})+v\kappa_{3,1},
\]
where \(\kappa_{3,b}=\int_0^te^{-R_3s}e^{-R_b(t-s)}\dd s\) for \(b=2\),
and \(\kappa_{3,1}=\int_0^te^{-R_3s}\dd s\).  Thus \(u\), and then \(v\),
are recovered, including when \(R_3=R_2\).

At level $n=4$, the labeled row separates the $2+2$ and $3+1$ collision
patterns that block counts would aggregate.  The explicit back-substitution
is recorded in Appendix~\ref{app:xi-n4}.

It remains to prove the continuity assertion in the product topology.  Let
$\Pi^x_b$ be the paintbox on $[b]$ with ranked frequencies $x$ and dust
$1-|x|$.  For every $\pi\in\mathcal S_b$, put
\[
 N_\pi(x)=\Pp\{\Pi^x_b=\pi\}.
\]
For $\pi\ne0_b$, this is the numerator of the integral in
\eqref{eq:xi-schweinsberg-rates}.  For $\sigma\in\mathcal S_b$, the
continuous function $F_\sigma$ in \eqref{eq:xi-paintbox-coarser} satisfies
\[
 F_\sigma=\sum_{\pi:\,\sigma\preceq\pi}N_\pi,
\]
where $\preceq$ denotes refinement.  M\"obius inversion on the finite
lattice $\mathcal S_b$ therefore expresses each $N_\pi$ as a finite linear
combination of continuous products of power sums.  In particular, $N_\pi$
is continuous on all of $\K$; no continuity of $|x|$ is needed.

For a nontrivial $\pi$, set $g_\pi=N_\pi/V_2$ on $\K\setminus\{0\}$.
Since the event
$\{\Pi^x_b=\pi\}$ forces one fixed pair to choose the same positive atom,
\begin{equation}
 0\leq N_\pi(x)\leq V_2(x),
 \qquad 0\leq g_\pi(x)\leq1.
 \label{eq:xi-kernel-bounded}
\end{equation}
We next identify its limit at the origin.  For the pattern consisting of one
binary merger and $s$ singleton blocks, condition on the atom chosen by the
merging pair.  A union bound over a singleton choosing that atom and over two
singletons choosing a common positive atom gives
\begin{equation}
 0\leq1-g_{(2;s)}(x)
 \leq s x_1+\binom{s}{2}V_2(x).
 \label{eq:xi-simple-binary-bound}
\end{equation}
If a merger group has size $k\geq3$, then
\[
 g_\pi(x)\leq\frac{V_k(x)}{V_2(x)}\leq x_1^{k-2}.
\]
If instead there are at least two merger groups and all are binary, then
\[
 g_\pi(x)
 \leq\frac{\sum_{i\ne j}x_i^2x_j^2}{V_2(x)}
 \leq V_2(x).
\]
Convergence to zero in the product topology is equivalent here to
$x_1\to0$, and $V_2(x)\leq x_1|x|\leq x_1$.  It follows that $g_\pi$
extends continuously to zero with value one for a single binary merger,
irrespective of the number of untouched singleton blocks, and value zero for
all other collision patterns.  These values agree with
Schweinsberg~\cite[Lemma~26]{Schweinsberg2000}.  With this extension, the full collision
rate, including the Kingman atom, is
\begin{equation}
 \lambda^\Xi_\pi=\int_\K g_\pi(x)\,\Xi(dx).
 \label{eq:xi-rate-continuous-kernel}
\end{equation}
Thus weak convergence of finite collision measures gives convergence of
every rate, every finite generator, and every finite endpoint row; compare
Schweinsberg~\cite[Proposition~27]{Schweinsberg2000}.

Write $F_t(\Xi)=\Law(P^\Xi(t))$.  If $\Xi_m\Rightarrow\Xi$ as finite
measures, take a subsequential weak limit of $F_t(\Xi_m)$, which exists
because $\cP(\K)$ is compact.  By
\eqref{eq:xi-paintbox-coarser} and \eqref{eq:Vs-tail}, every finite
paintbox restriction of this limit is the endpoint restriction generated
by $\Xi$.  Kingman's paintbox uniqueness \cite{Kingman1978} identifies
the limit with $F_t(\Xi)$.  Thus $F_t$ is continuous.

Two specified labels merge at rate $\Xi(\K)$, by
Schweinsberg~\cite[equation~(12)]{Schweinsberg2000}.  Conditional on the frequencies,
their probability of lying in the same block is $V_2(P)$, proving
\eqref{eq:xi-pair-identity}.  Now suppose that
$F_t(\Xi_m)\Rightarrow F_t(\Xi)$.  Since $V_2$ is continuous,
\[
 \Xi_m(\K)
 =-t^{-1}\log\left(1-\int_\K V_2(p)\,F_t(\Xi_m)(\dd p)\right)
 \longrightarrow\Xi(\K).
\]
The limiting argument of the logarithm is $e^{-t\Xi(\K)}>0$.
The measures $\Xi_m$ consequently have bounded total masses and are
relatively compact for weak convergence on compact $\K$.  Every
subsequential limit $\Xi'$ satisfies $F_t(\Xi')=F_t(\Xi)$ by continuity,
so injectivity gives $\Xi'=\Xi$.  Hence $\Xi_m\Rightarrow\Xi$.
Both spaces are metrizable, so this proves continuity of the inverse on
the image and the embedding assertion, including at $\Xi=0$.

Under unit-mass normalization, \eqref{eq:xi-pair-identity} recovers $t$.
More generally, the time-$t$ endpoint for $\Xi$ is the time-one endpoint
for $t\Xi$, proving the exact time--scale assertion.
\end{proof}

\begin{corollary}[Scalar recovery of the collision measure]
\label{cor:scalar-xi}
Let $X\in L^1$ be centered and have a distribution that is determining for
random mass partitions, and mark the blocks of the $\Xi$-coalescent by iid
copies of $X$.  For every known $t>0$, the map
\begin{equation}\label{eq:scalar-xi-observation}
 \Xi\longmapsto
 \Law\left(\sum_{j\ge1}P_j^\Xi(t)X_j\right)
\end{equation}
is a topological embedding from the finite nonnegative collision measures
on $\K$ into $\cP(\R)$, with both spaces equipped with weak convergence.
If $\Xi(\K)=1$, the same law
determines an unknown $t>0$.  Without normalization it determines exactly
$t\Xi$; for $\Xi=0$, time is not identifiable.  The conclusion applies to
the stable mark $X_*$ and, for every prescribed $q\in(0,\infty)$, to the
compound Poisson mark $X^{(q)}\in L^q$ of
Theorem~\ref{thm:compound-poisson}.
\end{corollary}

\begin{proof}
The observation map in \eqref{eq:scalar-xi-observation} is the composition
\[
 \Xi\longmapsto\Law(P^\Xi(t))
 \longmapsto T_X\bigl(\Law(P^\Xi(t))\bigr).
\]
The first map is an embedding on all finite collision measures by
Theorem~\ref{thm:xi-fixed-time}.  The second map, $T_X$, is a continuous
injection on compact $\cP(\K)$ and hence an embedding.  Their composition
is therefore an embedding.  If time is unknown,
\eqref{eq:xi-pair-identity} recovers it under unit mass.  In general, view
the recovered endpoint distribution as a time-one law and apply
Theorem~\ref{thm:xi-fixed-time} to recover the product $t\Xi$.
\end{proof}

\begin{corollary}[A scalar Fleming--Viot observation]
\label{cor:scalar-xifv}
Let $H$ be a known diffuse probability measure on a compact Polish space
$E$, and let $(M_t^\Xi)_{t\ge0}$ be the neutral $\Xi$-Fleming--Viot process,
without mutation or selection, started from $H$.  If $f:E\to\R$ is measurable and
$f_\#H=\Law(X)$ for a determining mark distribution, then, for every known
$t>0$, the observation
\[
 M_t^\Xi(f)=\int_E f(z)\,M_t^\Xi(\dd z)
\]
defines a topological embedding
$\Xi\mapsto\Law(M_t^\Xi(f))$ from the finite nonnegative collision
measures on $\K$ into $\cP(\R)$, with both spaces equipped with weak
convergence.  Under $\Xi(\K)=1$, the same law also determines an unknown
positive time.
\end{corollary}

\begin{proof}
The neutral lookdown construction and its de Finetti representation
\cite[Theorem~1.1, Corollary~2.8, and Section~5.1, especially
equation~(5.3)]{BirknerBlathMohleSteinruckenTams2009} give
\[
 M_t^\Xi\ \stackrel{d}{=}\
 \left(1-\sum_{j\ge1}P_j^\Xi(t)\right)H
 +\sum_{j\ge1}P_j^\Xi(t)\delta_{Z_j}.
\]
At time zero the exchangeable population has iid types with law $H$.
Conditionally on the reproduction structure, neutral reproduction only
copies parental types and distinct ancestral lineages select distinct
initial individuals.  Their types are therefore iid with law $H$ and
independent of that structure.  Consequently, conditionally on
$P^\Xi(t)$, the variables $Z_j$ in the display are iid with law $H$.
In this representation,
\[
 \E\left[M_t^\Xi(|f|)\mid P^\Xi(t)\right]
 =\int_E|f(z)|\,H(\dd z)=\E|X|<\infty,
\]
so $M_t^\Xi(f)$ is well defined almost surely.  Since a determining mark is
centered,
$M_t^\Xi(f)$ has the law in \eqref{eq:scalar-xi-observation}.  The claims
follow from Corollary~\ref{cor:scalar-xi}.
\end{proof}

\section{A uniform obstruction for finite-support marks}
\label{sec:finite-support}

The proof of Proposition~\ref{prop:intro-finite} is a localized form of the
projection obstruction in
Cuesta-Albertos et~al.~\cite[Theorem~3.5]{CuestaAlbertosFraimanRansford2007}.  We construct
$C_c^\infty$ densities arbitrarily close in total variation whose marked-sum
laws agree simultaneously for all mark distributions with the prescribed
support bound.
Put $M=d+1$ and use the fixed-total-mass slice
$\mathcal F_{M,s}^{\circ}$ defined in \eqref{eq:strict-face}.
In this section $T_X$ is understood from the same formula
\eqref{eq:TX-intro} for every real finite-support mark, whether centered or
not.

\begin{proof}[Proof of Proposition~\ref{prop:intro-finite}]
This slice has the affine chart
\begin{align}
 C_s&=\left\{y\in\R^d:
 y_1>\cdots>y_d>s-\sum_{j=1}^dy_j>0\right\},\label{eq:Cs}\\
 \iota_s(y)&=\left(y_1,\ldots,y_d,
 s-\sum_{j=1}^dy_j,0,0,\ldots\right).
 \label{eq:iota-s}
\end{align}
The chamber is nonempty.  For example, with $1\le i\le M$,
\begin{equation}\label{eq:strict-point}
 p_i=\frac{2s(M+1-i)}{M(M+1)}
\end{equation}
gives a strictly decreasing positive vector of total mass $s$.
The map $\iota_s$ is an affine diffeomorphism from $C_s$ onto
$\mathcal F_{M,s}^{\circ}$.

Let $U$ be as in the proposition.  Its inverse image under $\iota_s$ contains an
open ball $B$ whose closure is contained in that inverse image.  Choose a
nonzero real function $g\in C_c^\infty(B)$.

For $b=(b_1,\ldots,b_d)$, evaluate the Vandermonde polynomial with one
coordinate fixed at zero:
\begin{equation}\label{eq:W-polynomial}
 W(b)=\prod_{i=1}^db_i
      \prod_{1\le i<j\le d}(b_i-b_j),
 \qquad N:=\deg W=\frac{d(d+1)}2.
\end{equation}
For $a=(a_1,\ldots,a_M)\in\R^M$, put
\begin{equation}\label{eq:anchored-b}
 b(a)=(a_1-a_M,\ldots,a_d-a_M).
\end{equation}
If $a$ has at most $d$ distinct coordinates, two of its $M=d+1$
coordinates agree.  If $a_i=a_M$ for some $i\le d$, then the corresponding
factor $b_i(a)$ vanishes.  Otherwise two of the first $d$ coordinates agree,
and one factor $b_i(a)-b_j(a)$ vanishes.  Thus
\begin{equation}\label{eq:W-zero}
 W(b(a))=0
 \quad\text{whenever }a\text{ has at most }d\text{ distinct coordinates}.
\end{equation}

Let $D=(\partial_1,\ldots,\partial_d)$ and set
\begin{equation}\label{eq:h-multiplier}
 h=W(D)^2g.
\end{equation}
Then $h$ is real, smooth and supported in $B$.  With the Fourier convention
$\widehat f(\xi)=\int_{\R^d}e^{i\xi\cdot y}f(y)\dd y$,
\begin{equation}\label{eq:hhat}
 \widehat h(\xi)=(-1)^NW(\xi)^2\widehat g(\xi).
\end{equation}
The perturbation is nonzero.  Indeed, if $h\equiv0$, then the entire
Fourier transform $\widehat g$ would vanish on the nonempty Euclidean open
set $\{\xi\in\R^d:W(\xi)\ne0\}$.  Since the restriction of $\widehat g$
to $\R^d$ is real analytic, it would then vanish on the connected space
$\R^d$; injectivity of the Fourier transform would give $g=0$, a
contradiction.  Also
\begin{equation}\label{eq:h-zero-integral}
 \int_{\R^d}h(y)\dd y=\widehat h(0)=0.
\end{equation}

Choose $\psi\in C_c^\infty(B)$, $\psi\ge0$, which is strictly positive on
the compact set $K_h=\supp h$, and normalize
$f=\psi/\int\psi$.  Put
\[
 m=\min_{K_h}f>0,
 \qquad L=\|h\|_\infty>0.
\]
For
\[
 0<\varepsilon<
 \min\left\{\frac mL,\frac{\delta}{2\int_{\R^d}|h(y)|\dd y}\right\},
\]
the functions
\begin{equation}\label{eq:fpm}
 f_+=f+\varepsilon h,
 \qquad f_-=f-\varepsilon h
\end{equation}
are nonnegative: on $K_h$ this follows from the choice of $\varepsilon$, and
off $K_h$ they equal $f$.  They both integrate to one by
\eqref{eq:h-zero-integral}, and they differ because $h\ne0$.  Define
\begin{equation}\label{eq:Hpm}
 H_{d,s,\pm}=(\iota_s)_\#(f_\pm(y)\dd y).
\end{equation}
These are distinct probability laws compactly supported in $U$.
Their pullbacks under $\iota_s$ have the $C_c^\infty$ densities $f_\pm$,
and, since $\iota_s$ is a Borel isomorphism onto its image, pushforward
preserves total variation.  Hence
\[
 0<\|H_{d,s,+}-H_{d,s,-}\|_{\TV}
   =2\varepsilon\int_{\R^d}|h(y)|\dd y<\delta.
\]

It remains to prove the simultaneous observation identity.  For
$a\in\R^M$, the affine chart gives
\begin{equation}\label{eq:projection-chart}
 a\cdot\iota_s(y)=sa_M+b(a)\cdot y.
\end{equation}
If $a$ has at most $d$ distinct coordinates, then for every $t\in\R$,
\begin{align}
 &\int e^{it a\cdot\iota_s(y)}(f_+(y)-f_-(y))\dd y \notag\\
 &\quad=2\varepsilon e^{itsa_M}\widehat h(tb(a))\notag\\
 &\quad=2\varepsilon e^{itsa_M}(-1)^Nt^{2N}
        W(b(a))^2\widehat g(tb(a))=0.
 \label{eq:projection-identity}
\end{align}
Thus the pushforwards of $H_{d,s,+}$ and $H_{d,s,-}$ under
$p\mapsto a\cdot p$ coincide.

Now let $F$ be any real mark distribution supported on at most $d$ points.
For every realization
$a=(X_1,\ldots,X_M)$ of the iid mark vector, $a$ has at most $d$
distinct coordinates, so \eqref{eq:projection-identity} applies pointwise.
Integrating it with respect to $F^{\otimes M}$ gives, for every $t\in\R$,
\[
 \E_{H_{d,s,+}}\exp\left\{it\sum_{i=1}^MP_iX_i\right\}
 =
 \E_{H_{d,s,-}}\exp\left\{it\sum_{i=1}^MP_iX_i\right\}.
\]
Characteristic-function uniqueness proves the proposition.  The pair in
\eqref{eq:Hpm} depends only on $(d,s,U,\delta)$ and is chosen independently of the
atom locations and probabilities of $F$.
\end{proof}

The construction makes the polynomial obstruction and affine localization in
Lin~\cite[Lemma~2.3.5 and Proposition~2.3.6 in the arXiv version]{Lin2016}
explicit for the directions
produced by finitely supported marks.  The earlier Fourier construction of
B{\'e}lisle et~al.~\cite[Theorem~5.4 and Lemma~5.5]{BelisleMasseRansford1997} prescribes families
of coinciding projections together with moment bounds.  For mixture models
with conditional densities, Le et~al.~\cite[Theorem~1]{LeBarilettoRinaldoHo2026} use
adjoint differential operators to construct distinct mixing measures with the
same mixture density.

\begin{remark}[Centering and the case $d=1$]
Every finite-support mark is integrable.  If $m_X=\E X$ and
$X^\circ=X-m_X$, then on the slice of total mass $s$,
\begin{equation}\label{eq:centering-translation}
 \sum_iP_iX_i=\sum_iP_iX_i^\circ+s m_X.
\end{equation}
Thus centering translates both observations by the same deterministic
constant.  The fixed total mass is essential to this conclusion.  When
$d=1$, every allowed mark is a constant, so every input law concentrated on
partitions of the same total mass has the same constant output.
\end{remark}

\begin{corollary}\label{cor:no-finite-identifying}
No real mark distribution with finite support is determining for random mass
partitions on $\cP(\K)$.
\end{corollary}

\begin{proof}
For a one-point support, use the preceding remark.  If the support has
cardinality $d\ge2$, apply Proposition~\ref{prop:intro-finite} with that value of
$d$.
\end{proof}

For partitions of a fixed finite set, divide-and-color maps are finite
linear operators.  Their nontrivial kernels are studied in
Steif and Tykesson~\cite[Theorem~2.1]{SteifTykesson2019}, and their ranks and nullities,
including the permutation-invariant restriction, are computed by
Forsstr{\"o}m and Steif~\cite[Theorems~1, 2 and~5]{ForsstromSteif2021}.  Proposition~\ref{prop:intro-finite}
concerns distributions of mass partitions and gives indistinguishability
simultaneously for all atom locations and probabilities within the support
bound.

\section{The \texorpdfstring{$\Lambda$}{Lambda} subclass}
\label{sec:lambda}

The structural restriction to $\Lambda$-coalescents permits identification
both from an asymmetric Bernoulli coloring and from a class of unbounded
marks.  Both arguments recover successive Hausdorff moments of the collision
measure.

Let $\Lambda\in\cP([0,1])$, and let
$(\Pi^\Lambda(t))_{t\ge0}$ be the standard $\Lambda$-coalescent started
from the singleton partition, in the convention of
Pitman~\cite[Theorem~1 and formula~(1)]{Pitman1999}.  If there are $b$ blocks, each specified
merger of $k$ of them has rate
\begin{equation}\label{eq:lambda-rates}
 \lambda^\Lambda_{b;k}
 =\int_{[0,1]}x^{k-2}(1-x)^{b-k}\,\Lambda(\dd x),
 \qquad 2\le k\le b.
\end{equation}
Write $P^\Lambda(t)=(P_j^\Lambda(t))_{j\ge1}$ for the ranked positive
asymptotic block frequencies and
$d_t=1-\sum_jP_j^\Lambda(t)$ for their missing mass.

Fix $q\in(0,1)\setminus\{1/2\}$, let $(B_j)$ be iid
Bernoulli$(q)$ variables independent of the coalescent, and define
\begin{equation}\label{eq:lambda-coloured-frequency}
 Y_{\Lambda,q}(t)
 :=q d_t+\sum_{j\ge1}P_j^\Lambda(t)B_j
 =q+\sum_{j\ge1}P_j^\Lambda(t)(B_j-q).
\end{equation}
Thus $Y_{\Lambda,q}(t)\in[0,1]$.  The first expression in
\eqref{eq:lambda-coloured-frequency} colors dust at its base frequency; the
second expresses the same observable through a centered mark.

\begin{theorem}\label{thm:lambda-bernoulli}
For every known $t>0$, the map
\[
 \cP([0,1])\longrightarrow\cP(\R),
 \qquad
 \Lambda\longmapsto\Law(Y_{\Lambda,q}(t)),
\]
is a topological embedding for weak convergence.  Under the normalization
$\Lambda([0,1])=1$, the same law determines an unknown positive time $t$.
For a nonzero finite unnormalized collision measure $\Gamma$, an observation
at unknown time determines exactly $t\Gamma$, and hence only the equivalence
class
\begin{equation}\label{eq:lambda-scale}
 (t,\Gamma)\sim(t/c,c\Gamma),\qquad c>0.
\end{equation}
\end{theorem}

\begin{proof}
Put $b_n^{\Lambda,q}(t)=\E[Y_{\Lambda,q}(t)^n]$.  Coloring the blocks of
the ancestral restriction independently gives the paintbox identity
\begin{equation}\label{eq:lambda-colour-duality}
 b_n^{\Lambda,q}(t)=\E[q^{N_n^\Lambda(t)}],
\end{equation}
where $N_n^\Lambda(t)$ is the number of blocks in the restriction to $[n]$.
Thus the output law gives every $b_n^{\Lambda,q}(t)$.

Write $\mu_m=\int x^m\Lambda(\dd x)$.  Suppose that
$\mu_0,\ldots,\mu_{n-3}$ have already been determined, and set
$z=\mu_{n-2}$.  The total rate from $n$ blocks to $j<n$ blocks is
\begin{equation}\label{eq:lambda-block-count-rate}
 r_{n,j}=\binom n{j-1}
 \int_{[0,1]}x^{n-j-1}(1-x)^{j-1}\Lambda(\dd x).
\end{equation}
It is affine in $z$, with coefficient
\begin{equation}\label{eq:lambda-new-moment-coefficients}
 c_{n,j}=\binom n{j-1}(-1)^{j-1},
 \qquad
 c_n:=\sum_{j<n}c_{n,j}=(-1)^n(n-1).
\end{equation}
Every generator for a restriction with fewer than $n$ labels depends only
on the moments already known.

Let $\Lambda_0$ and $\Lambda_1$ have the same moments through order $n-3$,
and write
\[
 \Lambda_s=(1-s)\Lambda_0+s\Lambda_1,
 \qquad z_i=\mu_{n-2}(\Lambda_i),
 \qquad 0\le s\le1.
\]
Abbreviate $b_j^s(t)=b_j^{\Lambda_s,q}(t)$.  For $j<n$, these functions do
not depend on $s$.  The finite backward equation is
\begin{equation}\label{eq:lambda-backward}
 \frac{\dd}{\dd t}b_n^s(t)
 =\sum_{j=1}^{n-1}r_{n,j}^{\Lambda_s}b_j^s(t)
  -r_n^{\Lambda_s}b_n^s(t),
 \qquad b_n^s(0)=q^n,
\end{equation}
where $r_n=\sum_{j<n}r_{n,j}$.  The finite generator is affine in $s$, so
its matrix exponential, and hence $b_n^s(t)$, is differentiable in $s$.
For $v_s(t)=\partial_s b_n^s(t)$, differentiation of
\eqref{eq:lambda-backward} gives
\begin{equation}\label{eq:lambda-variation}
 v_s'(t)=-r_n^{\Lambda_s}v_s(t)
 +(z_1-z_0)H_{n,q}^{\Lambda_s}(t),
 \qquad v_s(0)=0,
\end{equation}
where substituting \eqref{eq:lambda-new-moment-coefficients} and carrying out
two binomial expansions yields
\begin{equation}\label{eq:lambda-forcing-linear}
 H_{n,q}^{\Lambda}(t)
 =\sum_{j=1}^{n-1}c_{n,j}b_j^{\Lambda,q}(t)
  -c_n b_n^{\Lambda,q}(t),
\end{equation}
and hence
\begin{equation}\label{eq:lambda-forcing}
 H_{n,q}^{\Lambda}(t)=\E\left[
 Y_{\Lambda,q}(t)(1-Y_{\Lambda,q}(t))^n
 +(-1)^nY_{\Lambda,q}(t)^n(1-Y_{\Lambda,q}(t))\right].
\end{equation}

To determine the sign, color the ancestral blocks of the restriction to
$[n+1]$ independently with probability $q$ for color one.  The expectations
$\E[Y(1-Y)^n]$ and $\E[Y^n(1-Y)]$, with $Y=Y_{\Lambda,q}(t)$, are the
probabilities of the patterns in which label $1$ has color one and all
remaining labels have color zero, or vice versa.
For the second term, exchangeability allows the exceptional label to be
chosen as label $1$.  Let $K$ be the number of ancestral blocks and $S_1$
the event that label $1$ is a singleton.  Both patterns are impossible
outside $S_1$.  Conditional on the ancestral partition, their probabilities
are, respectively,
\[
 \1_{S_1}q(1-q)^{K-1},
 \qquad \1_{S_1}(1-q)q^{K-1}.
\]
For even $n$, their sum is nonnegative.  The no-merger event has probability
$\exp\{-r_{n+1}^{\Lambda}t\}>0$, lies in $S_1$, and has $K=n+1$;
the sum is therefore strictly positive in expectation.  For odd $n$, their
difference gives
\begin{equation}\label{eq:lambda-odd-forcing}
 H_{n,q}^{\Lambda}(t)=\E\left[
 \1_{S_1}q(1-q)
 \{(1-q)^{K-2}-q^{K-2}\}\right].
\end{equation}
On $S_1$ we have $K\ge2$; the bracket in
\eqref{eq:lambda-odd-forcing} is zero when $K=2$ and has the sign of
$1-2q$ when $K>2$.  The no-merger event therefore makes this sign strict
in expectation.

Variation of constants in \eqref{eq:lambda-variation} gives
\begin{equation}\label{eq:lambda-variation-solution}
 v_s(t)=(z_1-z_0)\int_0^t
 \exp\{-r_n^{\Lambda_s}(t-r)\}
 H_{n,q}^{\Lambda_s}(r)\dd r.
\end{equation}
For every fixed positive $t$, the observed $n$th moment is therefore strictly
monotone in the new Hausdorff moment $\mu_{n-2}$ while the lower moments are
fixed.  For odd $n$, the orientation reverses when $q$ crosses $1/2$.
Starting at $n=3$ recovers all moments of $\Lambda$; the Hausdorff moment
theorem then recovers $\Lambda$.

For continuity, each finite restriction has a finite-state generator whose
entries are integrals of polynomials against $\Lambda$.  Hence weak
convergence of $\Lambda$ gives convergence of every quantity in
\eqref{eq:lambda-colour-duality}.  Since the observations are supported on
$[0,1]$, convergence of all their moments is equivalent to weak convergence.
The domain $\cP([0,1])$ is compact, so the continuous injection is a
topological embedding.

Finally, two blocks merge at rate one under the probability normalization.
The two-label restriction gives
\begin{equation}\label{eq:lambda-time}
 \E[Y_{\Lambda,q}(t)^2]
 =q(1-e^{-t})+q^2e^{-t}
 =q-q(1-q)e^{-t},
\end{equation}
which recovers $t$.  For a finite nonzero measure $\Gamma$, put
$m=\Gamma([0,1])$ and $\Lambda=\Gamma/m$.  The process at time $t$ driven
by $\Gamma$ is the process at time $tm$ driven by $\Lambda$.  Formula
\eqref{eq:lambda-time} first recovers $tm$, and the fixed-time argument then
recovers $\Lambda$.  The observation therefore recovers
$tm\Lambda=t\Gamma$ exactly, but cannot select a representative of
\eqref{eq:lambda-scale}.
\end{proof}

The recovery proceeds through the Hausdorff moments of $\Lambda$.
The dependence of a contemporaneous sample likelihood on the first $n-2$
moments is recorded in Koskela et~al.~\cite[Lemma~4]{KoskelaJenkinsSpano2018}.
Other genetic observations retain different information:
Spence et~al.~\cite{SpenceKammSong2016} obtain identifiability results in restricted
models from expected site-frequency spectra, while
Mir{\'o} Pina et~al.~\cite{MiroPinaJolySiriJegousse2023} estimate a $\Lambda$ density from
subsampled weighted site-frequency spectra under regularity assumptions.
The theorem above uses the complete distribution of the type frequency at
one time, started from a known initial frequency.

\begin{remark}
For $q>1/2$, one may equivalently use the reflection
$Y_{\Lambda,q}=1-Y'_{\Lambda,1-q}$ formed from $B_j'=1-B_j$.  At $q=1/2$
the odd forcing in \eqref{eq:lambda-odd-forcing} vanishes; the argument gives
neither identification nor non-identification at that value.
\end{remark}

The same induction also identifies $\Lambda$ from its absorption
probabilities at one time.  M{\"o}hle~\cite[Lemma~6.3]{Mohle2008} proves
identification from the full marginal laws of the absorption times for all
sample sizes.  The proposition below uses only the probability of absorption
by one fixed time for every finite sample size.  This provides the link to
marks whose laws are not finitely supported.

\begin{proposition}[Identification from absorption probabilities]
\label{prop:lambda-absorption}
For every $t>0$, the map
\[
 \Lambda\longmapsto
 \bigl(\E V_n(P^\Lambda(t))\bigr)_{n\ge2}
\]
is a topological embedding of $\cP([0,1])$ into
$[0,1]^{\{2,3,\ldots\}}$ with its product topology.
\end{proposition}

\begin{proof}
Set
\[
 a_1^\Lambda(u)=1,\qquad
 a_n^\Lambda(u)=\Pp(N_n^\Lambda(u)=1)
              =\E V_n(P^\Lambda(u)),\quad n\ge2.
\]
The convention $a_1^\Lambda=1$ refers to absorption of a one-label
restriction; it is not the expected total atomic mass, which may be less
than one.  The backward equations for $a_n^\Lambda$ have the same rates as
\eqref{eq:lambda-backward}, with initial values $a_n^\Lambda(0)=0$ for
$n\ge2$.

Use $\Lambda_s$, $z_i$, $c_{n,j}$ and $c_n$ as in the proof of
Theorem~\ref{thm:lambda-bernoulli}.  At the induction step, all generators
with fewer than $n$ blocks have already been recovered.  They determine
$a_j^{\Lambda_s}(u)$ for every $u\ge0$ and $j<n$, independently of $s$;
no observations at additional times are needed.  Differentiating the
backward equation and integrating gives
\begin{equation}\label{eq:lambda-absorption-variation}
 \begin{aligned}
 \partial_s a_n^{\Lambda_s}(t)
 &=(z_1-z_0)\int_0^t
 e^{-r_n^{\Lambda_s}(t-u)}H_n^{\Lambda_s}(u)\dd u,\\
 H_n^\Lambda(u)
 &=\sum_{j<n}c_{n,j}a_j^\Lambda(u)-c_na_n^\Lambda(u).
 \end{aligned}
\end{equation}
To determine the sign, let $S_1$ be the event that label $1$ is a singleton
in the restriction to $[n+1]$ at time $u$, and let $K$ be its number of
blocks.  Dividing \eqref{eq:lambda-colour-duality} and
\eqref{eq:lambda-forcing-linear} by $q$ and letting $q\downarrow0$ gives
\begin{equation}\label{eq:lambda-absorption-forcing}
 H_n^\Lambda(u)
 =\Pp(S_1)+(-1)^n\Pp(S_1,\,K=2).
\end{equation}
Indeed, $b_j^{\Lambda,q}(u)/q=\E[q^{N_j^\Lambda(u)-1}]\to a_j^\Lambda(u)$,
including $j=1$.  For the two color patterns used in the Bernoulli proof,
the respective limits after division by $q$ are
$\Pp(S_1)$ and $\Pp(S_1,K=2)$.
For odd $n$, the right-hand side of
\eqref{eq:lambda-absorption-forcing} is $\Pp(S_1,K>2)$; for even $n$
it is a sum of nonnegative probabilities.  In either case, for $n\ge3$,
\[
 H_n^\Lambda(u)\ge\Pp(K=n+1)=e^{-r_{n+1}^\Lambda u}>0.
\]
Thus $a_n^\Lambda(t)$ is strictly increasing in the new moment
$\mu_{n-2}$ when the lower moments are fixed.  Since $\mu_0=1$,
induction from $n=3$ recovers every moment of $\Lambda$, and the Hausdorff
moment theorem gives injectivity.  The finite generators depend continuously
on $\Lambda$, so each coordinate $a_n^\Lambda(t)$ is continuous.
Compactness of $\cP([0,1])$ proves the embedding assertion.
\end{proof}

\Needspace{8\baselineskip}
The absorption profile also has an interpretation in terms of a single
block.  Let $Z_t^\Lambda$ be the asymptotic frequency of the block containing
label $1$, with value zero when that label belongs to the dust.  By the
paintbox representation,
\[
 \Law(Z_t^\Lambda\mid P^\Lambda(t)=p)
 =\left(1-\sum_{j\ge1}p_j\right)\delta_0
   +\sum_{j\ge1}p_j\delta_{p_j}.
\]
Consequently,
\[
 \E[(Z_t^\Lambda)^{n-1}]=\E V_n(P^\Lambda(t)),\qquad n\ge2.
\]
Probability laws on $[0,1]$ are determined by their moments.  Thus
Proposition~\ref{prop:lambda-absorption} also identifies $\Lambda$ from the
law of $Z_t^\Lambda$ at any known positive time.

\Needspace{13\baselineskip}
For an integrable real mark $X$, write
\[
 a_X=\sup\{r>0:\E|X|^r<\infty\}
\]
for its moment abscissa.  The following result uses the power-sum convergence
principle of \cite[Theorem~2.2]{LenziBreiman} to pass from the marked law to the absorption
profile.  The
range $2\le a_X<4$ includes finite-variance marks: a fourth-order bound on
the nonlinear characteristic-function remainder extends the moment argument
to that range.

\begin{corollary}[Marks with a finite boundary moment]
\label{cor:lambda-boundary-mark}
Let $X$ be centered and nonconstant, with
\[
 1<a_X<4,\qquad \E|X|^{a_X}<\infty.
\]
For every known $t>0$, the map
\[
 \cP([0,1])\longrightarrow\cP(\R),\qquad
 \Lambda\longmapsto
 \Law\left(\sum_{j\ge1}P_j^\Lambda(t)X_j\right)
\]
is a topological embedding for the weak topologies, where the $X_j$ are
iid copies of $X$, independent of the coalescent.
Under the normalization $\Lambda([0,1])=1$, this law also determines an
unknown positive time $t$ together with $\Lambda$.
\end{corollary}

\begin{proof}
Suppose that $(t_0,\Lambda_0)$ and $(t_1,\Lambda_1)$, with $t_0,t_1>0$,
give the same marked-sum law.
This law is nondegenerate.  Indeed,
$\E V_2(P^\Lambda(t))=1-e^{-t}>0$ implies that the mass partition is
nonzero with positive probability.  Conditionally on any nonzero partition,
a positive multiple of $X_1$ is independent of the remaining summands,
so their sum cannot have a constant law.

Alternate the two laws $\Law(P^{\Lambda_0}(t_0))$ and
$\Law(P^{\Lambda_1}(t_1))$ along a sequence of random weights.  The associated
marked laws form a constant, nondegenerate sequence.  By
\cite[Theorem~2.2]{LenziBreiman}, all expected power sums of order greater
than one converge along this sequence.  Their two alternating values must
therefore agree.  At order two this yields
\[
 1-e^{-t_0}=\E V_2(P^{\Lambda_0}(t_0))
          =\E V_2(P^{\Lambda_1}(t_1))=1-e^{-t_1}.
\]
Hence $t_0=t_1$, and Proposition~\ref{prop:lambda-absorption} now gives
$\Lambda_0=\Lambda_1$.

For fixed $t$, continuity follows from Theorem~\ref{thm:xi-fixed-time}
restricted to the $\Lambda$ subclass and Lemma~\ref{lem:kernel-continuity}.
The domain $\cP([0,1])$ is compact, so the injection is a topological
embedding.
\end{proof}

\section{A finite-type \texorpdfstring{$\Xi$}{Xi}-Fleming--Viot obstruction}
\label{sec:xifv}

For $d\ge2$, let
\[
 \Sigma_d=\left\{y\in[0,1]^d:\sum_{k=1}^dy_k=1\right\}
\]
be the simplex of type frequencies.
Proposition~\ref{prop:intro-finite} gives two different distributions of
family sizes that cannot be distinguished by any mark with at most $d$
values.  We use this pair to construct two collision measures with the same
neutral $d$-type Fleming--Viot transition semigroup.

\begin{theorem}
\label{thm:xifv-obstruction}
For every \(d\geq2\) and \(\rho\in(0,1]\), there are distinct probability
collision measures \(\Xi_{d,\rho,+}\) and \(\Xi_{d,\rho,-}\), supported on
the fixed-total-mass slice $\mathcal F_{d+1,\rho}^{\circ}$ in
\eqref{eq:strict-face},
whose neutral \(d\)-type \(\Xi\)-Fleming--Viot
transition semigroups without mutation or selection agree:
\[
 \mathsf S_t^{\Xi_{d,\rho,+}}f(y)=\mathsf S_t^{\Xi_{d,\rho,-}}f(y)
\]
for every \(t\geq0\), \(y\in\Sigma_d\), and \(f\in C(\Sigma_d)\).
\end{theorem}

\begin{proof}
Choose $H_{d,\rho,+}$ and $H_{d,\rho,-}$ as in
Proposition~\ref{prop:intro-finite}, supported on
$\mathcal F_{d+1,\rho}^{\circ}$.
Let \((R_i)\) be centered Rademacher marks.  Conditional on a mass
partition \(p\),
\[
 \mathbb E\left[\left(\sum_ip_iR_i\right)^2\middle|p\right]=V_2(p).
\]
Equality of the marked-sum laws in
Proposition~\ref{prop:intro-finite} therefore gives
\begin{equation}
 c_{d,\rho}:=\int V_2\,\dd H_{d,\rho,+}
 =\int V_2\,\dd H_{d,\rho,-}>0.
 \label{eq:xifv-common-normalizer}
\end{equation}
Indeed, on this slice
\(\rho^2/(d+1)\leq V_2(p)\leq\rho^2\).  Define
\begin{equation}
 \Xi_{d,\rho,\pm}(dp)=\frac{V_2(p)}{c_{d,\rho}}H_{d,\rho,\pm}(dp).
 \label{eq:xifv-tilt}
\end{equation}
These are probability measures.  They are distinct: equality would imply
\(H_{d,\rho,+}=H_{d,\rho,-}\) after multiplying the common measure by
\(c_{d,\rho}/V_2\), which is bounded on this slice.  They also satisfy
\begin{equation}
 \frac{\Xi_{d,\rho,\pm}(dp)}{V_2(p)}=
 \frac{H_{d,\rho,\pm}(dp)}{c_{d,\rho}}.
 \label{eq:xifv-effective-measure}
\end{equation}
In particular,
\[
 \int_{\K\setminus\{0\}}\frac{\Xi_{d,\rho,\pm}(dp)}{V_2(p)}
 =\frac1{c_{d,\rho}}<\infty.
\]
There is no atom at the origin and hence no Kingman component.

For \(y\in\Sigma_d\), draw iid parent types \(K_i\) with
\(\mathbb P_y(K_i=k)=y_k\), and let $e_k$ be the $k$th standard basis vector
of $\R^d$.  A reproduction event with family sizes \(x\)
sends the state to
\begin{equation}
 \Phi(x,K,y)=(1-|x|)y+\sum_i x_i e_{K_i}.
 \label{eq:xifv-event-map}
\end{equation}
For \(\theta\in\mathbb R^d\),
\begin{equation}
 \theta\cdot\Phi(x,K,y)
 =\theta\cdot y+\sum_i x_i
 \bigl(\theta_{K_i}-\theta\cdot y\bigr).
 \label{eq:xifv-centred-bridge}
\end{equation}
The marks in parentheses are iid, centered under the current state \(y\),
and have support of cardinality at most \(d\).
Proposition~\ref{prop:intro-finite}, applied separately for each
\((\theta,y)\), and the Cram\'er--Wold theorem
show that the full post-event distributions in \eqref{eq:xifv-event-map}
agree after mixing under \(H_{d,\rho,+}\) and \(H_{d,\rho,-}\).  Define
\begin{equation}
 J_\pm f(y)=\int H_{d,\rho,\pm}(dx)\,\mathbb E_y[f(\Phi(x,K,y))],
 \qquad J_+=J_-=:J.
 \label{eq:xifv-common-kernel}
\end{equation}
On the polynomial core, the \(\Xi\)-Fleming--Viot generator formula
\cite[Proposition~1.3 and equation~(1.12)]{BirknerBlathMohleSteinruckenTams2009}
and \eqref{eq:xifv-effective-measure} give
\begin{equation}
 A_\pm f(y)=\frac{1}{c_{d,\rho}}\{Jf(y)-f(y)\}.
 \label{eq:xifv-generator}
\end{equation}
The map \(J\) preserves continuity: only \(d+1\) family coordinates are
nonzero, so the parent expectation is a finite sum with probabilities
polynomial in \(y\).  It is a Markov contraction, and
\(\|A_\pm\|\leq2/c_{d,\rho}\).  Hence the displayed core operator extends
to the same bounded operator on \(C(\Sigma_d)\); the polynomial class is a
core by Birkner et~al.~\cite[Proposition~4.3]{BirknerBlathMohleSteinruckenTams2009}.  In the present
finite-intensity setting, both processes therefore have the same uniformly
continuous transition semigroup,
\[
 \mathsf S_t=e^{tA}=e^{-t/c_{d,\rho}}
 \sum_{m\geq0}\frac{(t/c_{d,\rho})^m}{m!}J^m.
\]
\end{proof}

For comparison, Der and Plotkin~\cite{DerPlotkin2014} showed that, under unit-mass
normalization and positive mutation rates, a stationary two-allele
$\Lambda$-Fleming--Viot distribution need not determine the drift measure,
even though it determines the mutation rates.  The theorem above gives
equality of the entire mutation-free transition semigroup for distinct
probability collision measures.

\begin{remark}
For $d=2$, the block-counting chains associated with the two measures in
Theorem~\ref{thm:xifv-obstruction} have the same transition semigroup.
Write $\Xi_\pm=\Xi_{2,\rho,\pm}$, let $Y_t$ be the first-type frequency
started from $q\in[0,1]$, and let $N_n(t)$ be the number of blocks in
the coalescent restriction to $[n]$ started from singletons.
Moment duality
\cite[Section~5.3]{BirknerBlathMohleSteinruckenTams2009} gives
\[
 \E_q^{\Xi_\pm}[Y_t^n]
 =\E^{\Xi_\pm}[q^{N_n(t)}]
 =\sum_{k=1}^n\Pp^{\Xi_\pm}(N_n(t)=k)q^k.
\]
Theorem~\ref{thm:xifv-obstruction} makes the left-hand sides equal for
every $q\in[0,1]$, $n\ge1$, and $t\ge0$.  Equality of polynomial
coefficients therefore gives all transition probabilities of the
block-counting chains.  Since these are time-homogeneous Markov chains,
their path laws agree from every finite initial number of blocks.
In particular, their times to the most recent common ancestor have the
same laws for every finite sample size.  Their full ranked-frequency laws
differ at every positive time by Theorem~\ref{thm:xi-fixed-time}.
M{\"o}hle~\cite[Section~7]{Mohle2008} constructs distinct $\Xi$-coalescents with
the same total collision rates; the pair above has the same entire
block-counting transition semigroup.
\end{remark}

\appendix
\section{Determining compound Poisson marks with prescribed moments}
\label{app:compound-poisson}

The stable mark in Section~\ref{sec:stable-mark} has a fixed moment boundary.  We
prove Theorem~\ref{thm:compound-poisson}, showing that the boundary can be
placed above any prescribed finite order.
The resulting characteristic exponent contains an analytic remainder.  Its
even integer powers must be separated from the fractional powers before the
argument of Lemma~\ref{lem:profile-extraction} can be reused.

The mark will in fact be chosen with some $a>q$ such that
\begin{equation}\label{eq:CP-moment-boundary}
 \E|X^{(q)}|^a<\infty,
 \qquad
 \E|X^{(q)}|^s=\infty\quad(s>a).
\end{equation}

\subsection{The mark and its characteristic exponent}

Choose an integer $M\ge1$ such that $q<2M+2$, a rational number $a$ and a
positive rational number $\varepsilon$ satisfying
\begin{equation}\label{eq:CP-parameters}
 \max\{q,2M\}<a<2M+2,
 \qquad 0<\varepsilon<2M+2-a.
\end{equation}
Put
\begin{equation}\label{eq:CP-alpha-d}
 \alpha_r=a+\varepsilon e^{-r},
 \qquad d_r=(\alpha_r-a)2^{-r},
 \qquad r\ge1.
\end{equation}
Then $2M<a<\alpha_r<2M+2$, $\alpha_r\downarrow a$, and
\begin{equation}\label{eq:CP-boundary-sum}
 \sum_{r\ge1}\frac{d_r}{\alpha_r-a}=\sum_{r\ge1}2^{-r}<\infty.
\end{equation}
The set $\{2,\alpha_1,\alpha_2,\ldots\}$ is linearly independent over
$\mathbb Q$.  Indeed, a finite relation
$2q_0+\sum_{r=1}^Rq_r\alpha_r=0$ becomes
\[
 2q_0+a\sum_{r=1}^Rq_r+
 \varepsilon\sum_{r=1}^Rq_rz^r=0,
 \qquad z=e^{-1}.
\]
Transcendence of $z$ first gives $q_r=0$ for $r\ge1$ and then $q_0=0$.

Consider the symmetric measure
\begin{equation}\label{eq:CP-Levy-measure}
 \Pi(\dd x)=\sum_{r\ge1}d_r|x|^{-1-\alpha_r}
              \1_{\{|x|>1\}}\dd x.
\end{equation}
It is finite, since
\begin{equation}\label{eq:CP-intensity}
 \Pi(\R)=2\sum_{r\ge1}\frac{d_r}{\alpha_r}<\infty.
\end{equation}
Let $N$ have the Poisson distribution with mean $\Pi(\R)$, let $(J_i)$ be
iid with law
$\Pi/\Pi(\R)$, independently of $N$, and set
\begin{equation}\label{eq:CP-X}
 X^{(q)}=\sum_{i=1}^N J_i.
\end{equation}
This is a nonconstant symmetric compound Poisson random variable.  Moreover,
\begin{equation}\label{eq:CP-a-moment-Levy}
 \int_\R|x|^a\Pi(\dd x)
 =2\sum_{r\ge1}\frac{d_r}{\alpha_r-a}<\infty.
\end{equation}
Conditionally on $N=n$,
\[
 \left|\sum_{i=1}^nJ_i\right|^a
 \le n^{a-1}\sum_{i=1}^n|J_i|^a.
\]
All moments of $N$ are finite, so \eqref{eq:CP-a-moment-Levy} implies
$X^{(q)}\in L^a\subset L^q$.  Symmetry and $a>1$ give
$\E X^{(q)}=0$.  If $s>a$, choose $r$ with $\alpha_r<s$.  Then the jump law
has infinite $s$th moment, and on the event $\{N=1\}$ the compound Poisson
sum consists of a single jump.  Hence $\E|X^{(q)}|^s=\infty$.  This proves the moment
claims of the theorem.

We need an exact decomposition of
\begin{equation}\label{eq:h-alpha}
 h_\alpha(u)=\int_1^\infty(1-\cos ux)x^{-1-\alpha}\dd x,
 \qquad 2M<\alpha<2M+2.
\end{equation}
Define
\begin{align}
 P_M(y)&=\sum_{k=1}^M\frac{(-1)^{k+1}y^{2k}}{(2k)!},\notag\\
 R_M(y)&=1-\cos y-P_M(y),\notag\\
 C_\alpha&=\int_0^\infty R_M(y)y^{-1-\alpha}\dd y.
 \label{eq:C-alpha-def}
\end{align}
The last integral is absolutely convergent: $R_M(y)=O(y^{2M+2})$ at zero
and is a polynomial of degree at most $2M$ plus a bounded function at
infinity.

\begin{lemma}[Fractional term and entire remainder]
\label{lem:fractional-entire}
For $2M<\alpha<2M+2$,
\begin{equation}\label{eq:h-decomposition}
 h_\alpha(u)=C_\alpha|u|^\alpha+A_\alpha(u),
\end{equation}
where
\begin{align}
 C_\alpha
 &=\frac{\pi}{2\Gamma(1+\alpha)\sin(\pi\alpha/2)}\ne0,
 \label{eq:C-alpha-exact}\\
 A_\alpha(u)
 &=\sum_{\ell\ge1}
 \frac{(-1)^{\ell+1}u^{2\ell}}
 {(2\ell)!(\alpha-2\ell)}.
 \label{eq:A-alpha-series}
\end{align}
The function $A_\alpha$ is even and entire.  On each compact subinterval of
$(2M,2M+2)$, the family $(A_\alpha)$ is locally uniformly bounded on
$\mathbb C$.  The sign of $C_\alpha$ is $(-1)^M$.
\end{lemma}

\begin{proof}
Scaling the compensated integral gives
\[
 \int_0^\infty\{1-\cos(ux)-P_M(ux)\}x^{-1-\alpha}\dd x
 =C_\alpha|u|^\alpha.
\]
All boundary terms vanish in $2M$ integrations by parts.  At the $j$th
stage, the boundary product has the form
$R_M^{(j)}(y)y^{-\alpha+j}$; it is
$O(y^{2M+2-\alpha})$ at zero and
$O(y^{2M-\alpha})$ at infinity.  The two exponents have the required signs
because $2M<\alpha<2M+2$.  Since
\[
 R_M^{(2M)}(y)=(-1)^M(1-\cos y),
\]
putting $\beta=\alpha-2M\in(0,2)$ yields
\begin{equation}\label{eq:C-after-parts}
 C_\alpha=
 \frac{(-1)^M}{\alpha(\alpha-1)\cdots(\alpha-2M+1)}
 \int_0^\infty(1-\cos y)y^{-1-\beta}\dd y.
\end{equation}
Using
\[
 y^{-1-\beta}=\frac1{\Gamma(1+\beta)}
 \int_0^\infty t^\beta e^{-ty}\dd t
\]
and
$\int_0^\infty(1-\cos y)e^{-ty}\dd y=\{t(1+t^2)\}^{-1}$,
Tonelli's theorem gives
\begin{align*}
 \int_0^\infty(1-\cos y)y^{-1-\beta}\dd y
 &=\frac1{\Gamma(1+\beta)}
   \int_0^\infty\frac{t^{\beta-1}}{1+t^2}\dd t\\
 &=\frac{\pi}
 {2\Gamma(1+\beta)\sin(\pi\beta/2)}.
\end{align*}
Using
$\alpha(\alpha-1)\cdots(\alpha-2M+1)
=\Gamma(1+\alpha)/\Gamma(1+\beta)$ and
$\sin(\pi\alpha/2)=(-1)^M\sin(\pi\beta/2)$ in
\eqref{eq:C-after-parts} proves \eqref{eq:C-alpha-exact} and its sign.

To identify the remainder, split the compensated integral at one:
\begin{align*}
 A_\alpha(u)
 &=-\int_0^1\{1-\cos(ux)-P_M(ux)\}x^{-1-\alpha}\dd x\\
 &\quad+\int_1^\infty P_M(ux)x^{-1-\alpha}\dd x.
\end{align*}
Uniformly for $u$ in a compact set,
$R_M(ux)=O(x^{2M+2})$ as $x\downarrow0$; hence the compensated integrand is
dominated there by a constant times $x^{2M+1-\alpha}$, which is integrable
because $\alpha<2M+2$.
Termwise integration of the cosine series in the first integral, and of the
finite polynomial in the second, gives \eqref{eq:A-alpha-series}.  The
factorial denominator yields an infinite radius of convergence.  If $\alpha$
stays in a
compact subinterval between the consecutive even integers $2M$ and
$2M+2$, all denominators are uniformly separated from zero for the finitely
many nearby indices, while factorial decay controls the rest.  This proves
the asserted local uniformity.
\end{proof}

The L\'evy exponent of \eqref{eq:CP-X} is
\begin{equation}\label{eq:CP-psi}
 \psi(u)=-\log\phi_{X^{(q)}}(u)
 =2\sum_{r\ge1}d_rh_{\alpha_r}(u).
\end{equation}
Set
\begin{equation}\label{eq:beta-gamma}
 \beta_r=2d_rC_{\alpha_r},
 \qquad
 A(u)=2\sum_{r\ge1}d_rA_{\alpha_r}(u)
     =\sum_{\ell\ge1}\gamma_\ell u^{2\ell}.
\end{equation}
All $\alpha_r$ belong to the compact interval $[a,\alpha_1]$ contained in
$(2M,2M+2)$.  Formula~\eqref{eq:C-alpha-exact} shows that
$\sup_r|C_{\alpha_r}|<\infty$.  Since $\sum_rd_r<\infty$, the sums defining
the fractional part and $A$ converge locally uniformly, and $A$ is even and
entire.  Local uniform convergence on
complex discs permits termwise extraction of the Taylor coefficients by
Cauchy's formula, and the Taylor series of the resulting entire function
converges absolutely on every closed disc.  Consequently, for every
$0\le R<\infty$,
\begin{equation}\label{eq:CP-coefficient-sums}
 \sum_{r\ge1}|\beta_r|R^{\alpha_r}<\infty,
 \qquad
 \sum_{\ell\ge1}|\gamma_\ell|R^{2\ell}<\infty.
\end{equation}
Therefore
\begin{equation}\label{eq:CP-psi-decomposed}
 \psi(u)=\sum_{r\ge1}\beta_r|u|^{\alpha_r}
        +\sum_{\ell\ge1}\gamma_\ell u^{2\ell}.
\end{equation}
When $M$ is odd, the coefficients $\beta_r$ are negative.  Although some
individual coefficients are then negative, the complete expression
\eqref{eq:CP-psi-decomposed} is the nonnegative L\'evy exponent, and absolute
convergence justifies the subsequent separations.

\subsection{Separation of the fractional profile}

For $p\in\K$, summing \eqref{eq:CP-psi-decomposed} over the weights gives
\begin{equation}\label{eq:CP-kernel}
 k_u^{X^{(q)}}(p)=\exp\left\{
 -\sum_{r\ge1}\beta_r|u|^{\alpha_r}V_{\alpha_r}(p)
 -\sum_{\ell\ge1}\gamma_\ell u^{2\ell}V_{2\ell}(p)
 \right\}.
\end{equation}
All interchanges of sums are justified by \eqref{eq:CP-coefficient-sums} and
$0\le V_s\le1$.  Also, $a>2$ gives
$\int x^2\Pi(\dd x)<\infty$ and
\[
 \psi(v)\le\frac{v^2}{2}\int x^2\Pi(\dd x),
\]
so $\sum_j\psi(up_j)<\infty$.  The uniform power-sum tail
\eqref{eq:Vs-tail} and \eqref{eq:CP-coefficient-sums} also show directly
that $k_u^{X^{(q)}}\in C(\K)$.

Suppose that $\mu,\mu'\in\cP(\K)$ give the same output law and put
$\sigma=\mu-\mu'$.  Define the fractional profile associated with the
present exponents by
\begin{equation}\label{eq:CP-profile}
 v^{(q)}(p)=\bigl(V_{\alpha_r}(p)\bigr)_{r\ge1}.
\end{equation}
Fix $0<t_0<1$.  For finite-support multiindices
\[
 m=(m_r)\in\mathbb N_0^{(\mathbb N)},
 \qquad
 k=(k_\ell)\in\mathbb N_0^{(\mathbb N)},
\]
write
\begin{align*}
 |m|&=\sum_rm_r,& m!&=\prod_rm_r!,&
 \beta^m&=\prod_r\beta_r^{m_r},\\
 |k|&=\sum_\ell k_\ell,& k!&=\prod_\ell k_\ell!,&
 \gamma^k&=\prod_\ell\gamma_\ell^{k_\ell},\\
 N(k)&=\sum_{\ell\ge1}\ell k_\ell,&
 \lambda(m,k)&=m\mathbin{\cdot}\alpha+2N(k).
\end{align*}
We only need to isolate the coefficients with $N(k)=0$; recovering the
individual multiindices of the analytic remainder is unnecessary.
Expanding the two exponentials in \eqref{eq:CP-kernel} at
$u=t_0e^{-s}$ produces the coefficients
\begin{align}
 B_{m,k}
 &=\frac{(-1)^{|m|+|k|}\beta^m\gamma^k
 t_0^{\lambda(m,k)}}{m!k!}\notag\\
 &\quad\times
 \int_\K
 \prod_rV_{\alpha_r}(p)^{m_r}
 \prod_\ell V_{2\ell}(p)^{k_\ell}\,\sigma(\dd p).
 \label{eq:CP-Bmk}
\end{align}
This expansion is absolutely summable, since
\begin{equation}\label{eq:CP-TV-bound}
 \sum_{m,k}|B_{m,k}|
 \le\|\sigma\|_{\TV}\exp\left\{
 \sum_r|\beta_r|t_0^{\alpha_r}
 +\sum_\ell|\gamma_\ell|t_0^{2\ell}
 \right\}<\infty.
\end{equation}
Thus
\begin{equation}\label{eq:CP-eta}
 \eta=\sum_{m,k}B_{m,k}\delta_{\lambda(m,k)}
\end{equation}
is a finite signed measure, where terms corresponding to the same atom are
combined.
The coefficient at $(m,k)=(0,0)$ is zero because $\sigma(\K)=0$.
Equality of output characteristic functions becomes
\[
 \int_{[0,\infty)}e^{-s\lambda}\eta(\dd\lambda)=0,
 \qquad s>0.
\]
By Lemma~\ref{lem:signed-laplace}, $\eta=0$, despite the finite accumulation of
the fractional exponents.

Rational independence of $\{2,\alpha_1,\alpha_2,\ldots\}$ gives exactly
\begin{equation}\label{eq:CP-pair-separation}
 m\mathbin{\cdot}\alpha+2N
 =m'\mathbin{\cdot}\alpha+2N'
 \quad\Longrightarrow\quad m=m',\quad N=N'.
\end{equation}
It does not determine the multiindex $k$ arising from the analytic remainder:
several such multiindices may have the same weighted degree $N(k)$.  At
$\lambda_m=m\mathbin{\cdot}\alpha$, with $m\ne0$,
\eqref{eq:CP-pair-separation} forces the pair $(m,N=0)$, and
$N(k)=0$ forces $k=0$.  Therefore
\begin{equation}\label{eq:CP-pure-atom}
 0=\eta(\{\lambda_m\})
 =\frac{(-1)^{|m|}\beta^m t_0^{m\cdot\alpha}}{m!}
 \int_\K\prod_rV_{\alpha_r}(p)^{m_r}\,\sigma(\dd p).
\end{equation}
Every $\beta_r$ is nonzero by \eqref{eq:C-alpha-exact}, so all mixed moments
of the fractional profile \eqref{eq:CP-profile} agree under $\mu$ and
$\mu'$.  Hence the two pushforwards by $v^{(q)}$ have the same integrals
against every cylinder polynomial.  Such polynomials form a unital
point-separating algebra on the compact cube $[0,1]^{\N}$, so
Stone--Weierstrass gives
\begin{equation}\label{eq:CP-profile-law}
 (v^{(q)})_\#\mu=(v^{(q)})_\#\mu'.
\end{equation}

The indices $\alpha_r$ are distinct and converge to $a>1$.
Lemma~\ref{lem:profile-injective}, applied with $s_r=\alpha_r$, therefore
shows that $v^{(q)}$ is a homeomorphism of $\K$ onto its compact image.
Applying its inverse to \eqref{eq:CP-profile-law} gives
\[
 \mu=((v^{(q)})^{-1})_\#(v^{(q)}_\#\mu)
     =((v^{(q)})^{-1})_\#(v^{(q)}_\#\mu')=\mu'.
\]

This proves injectivity in Theorem~\ref{thm:compound-poisson}.  Affinity and
weak continuity follow from Lemma~\ref{lem:kernel-continuity}, since the mark is
centered and integrable.  Compactness of $\cP(\K)$ then makes the continuous
injection a topological embedding.

\section{The four-label inversion}\label{app:xi-n4}

The first level at which simultaneous collision patterns with the same
number of resulting blocks must be distinguished is $n=4$.  In the notation
of Section~\ref{sec:xi-fixed}, put
\begin{equation}\label{eq:xi-n4-rates}
 \begin{aligned}
  u_4&=\lambda_{4;2;2}^\Xi,&
  w_4&=\lambda_{4;2,2;0}^\Xi,\\
  v_4&=\lambda_{4;3;1}^\Xi,&
  z_4&=\lambda_{4;4;0}^\Xi,\\
  R_4&=6u_4+3w_4+4v_4+z_4.&&
 \end{aligned}
\end{equation}
The $14$ states in $D_4$ consist of six labeled $2+1+1$ states, three
labeled $2+2$ states, four labeled $3+1$ states, and the one-block state.
The subgenerator $B_4$ is known from the $n=2,3$ steps.  If $A$ is a
labeled three-block state, a path from $0_4$ ending at $A$ has exactly one
jump, directly from $0_4$: after another jump it has fewer than three blocks
and cannot return to $A$.  Hence
\begin{equation}\label{eq:xi-n4-three-block}
 p_4(t;A)=u_4\kappa_{4,3},
 \qquad
 \kappa_{4,b}=\int_0^te^{-R_4s}e^{-R_b(t-s)}\dd s.
\end{equation}
For each pattern, write $p_{211}$, $p_{22}$ and $p_{31}$ for the endpoint
probability of one specified labeled partition of that pattern; these are
not sums over partitions of the same type.  Let
$u_3=\lambda_{3;2;1}^\Xi$ be the rate of merging one specified pair at level
three, already recovered at the preceding step, and define
\begin{equation}\label{eq:xi-n4-path-integral}
 L_{432}(t)=\int_0^te^{-R_4s}
 \int_0^{t-s}e^{-R_3v}e^{-R_2(t-s-v)}\dd v\,\dd s.
\end{equation}
A specified $2+2$ partition has two three-block predecessors, corresponding
to which pair merged first.  A specified $3+1$ partition has three, one for
each pair within its triple.  From each predecessor the final merger has
rate $u_3$.  Consequently,
\begin{equation}\label{eq:xi-n4-two-block}
 \begin{aligned}
 u_4&=\frac{p_{211}}{\kappa_{4,3}},\\
 w_4&=\frac{p_{22}-2u_4u_3L_{432}(t)}{\kappa_{4,2}},\\
 v_4&=\frac{p_{31}-3u_4u_3L_{432}(t)}{\kappa_{4,2}},\\
 z_4&=R_4-6u_4-3w_4-4v_4.
 \end{aligned}
\end{equation}
The integral definitions remain valid when some exit rates coincide.  The
labeled endpoint row thus separates collision patterns that the block-count
distribution combines.

\section*{Acknowledgments}
The author is grateful to Fortunato Fulvio Bitonto for his patience and his happy heart.
He is also grateful to Irene Crimaldi and Pietro Rigo for their kindness.

OpenAI's ChatGPT was used for language editing and \LaTeX{} formatting, and for assistance
in refining the literature search to focus on work relevant to this paper;
responsibility for the mathematical content rests with the author.

\makeatletter
\def\@biblabel#1{\@defaultbiblabelstyle{#1}}
\def\bibsetup{}
\makeatother
\providecommand{\bysame}{\leavevmode\hbox to3em{\hrulefill}\thinspace}

\end{document}